\documentclass[11pt,a4paper]{amsart}
\usepackage{amssymb, amsmath, amscd, mathrsfs,marvosym, psfrag,color} 

\input xy
\xyoption{all}
\DeclareMathAlphabet{\mathpzc}{OT1}{pzc}{m}{it}

\newtheorem{theorem}{Theorem}[section]

\newtheorem{proposition}[theorem]{Proposition}

\newtheorem{lemma}[theorem]{Lemma}

\theoremstyle{definition}
\newtheorem{definition}[theorem]{Definition}

\theoremstyle{remark}
\newtheorem{remark}[theorem]{Remark}

\newcommand{\CO}{{\mathcal O}}

\newcommand{\CX}{{\mathcal X}}

\newcommand{\SA}{{\mathscr A}}
\newcommand{\SB}{{\mathscr B}}

\newcommand{\SF}{{\mathscr F}}

\newcommand{\SK}{{\mathscr K}}

\newcommand{\SR}{{\mathscr R}}

\newcommand{\SZ}{{\mathscr Z}}

\newcommand{\fp}{{{\mathfrak p}}}

\newcommand{\hE}{{\widehat E}}
\newcommand{\hF}{{\widehat F}}

\newcommand{\hM}{{\widehat M}}

\newcommand{\hf}{{\widehat f}}

\newcommand{\tM}{{\widetilde{M}}}

\newcommand{\tS}{{\widetilde{S}}}

\newcommand{\tf}{{\widetilde{f}}}

\newcommand{\DC}{{\mathbb C}}

\newcommand{\DZ}{{\mathbb Z}}

\newcommand{\DQ}{{\mathbb Q}}

\newcommand{\DF}{{\mathbb F}}

\newcommand{\End}{{\operatorname{End}}}
\newcommand{\Ext}{{\operatorname{Ext}}}
\newcommand{\Hom}{{\operatorname{Hom}}}

\newcommand{\rad}{{\operatorname{rad}}}

\newcommand{\im}{{\operatorname{im}\,}}

\newcommand{\rk}{{{\operatorname{rk}}}}

\newcommand{\ol}{\overline}

\newcommand{\id}{{\operatorname{id}}}

\newcommand{\Quot}{{\operatorname{Quot\,}}}

\newcommand{\fin}{{\text{fin}}}
\newcommand{\tilt}{{\text{tilt}}}

\newcommand{\comment}[1]{}

\newcommand{\lgl}{\langle}
\newcommand{\rgl}{\rangle}
\newcommand{\pre}{{\operatorname{pre}}}
\newcommand{\qchoose}[2]{\left[{#1\atop#2}\right]}

\begin{document}

\pagenumbering{arabic}
\title[]{Quantum tilting modules over local rings} \author[]{Peter Fiebig}
\begin{abstract} We show that tilting modules for quantum groups over local Noetherian domains of quantum characteristic 0 exist and that the indecomposable tilting modules are parametrized by their highest weight. For this we introduce a model category $\CX=\CX_\SA(R)$ associated with a Noetherian $\DZ[v,v^{-1}]$-domain $\SA$ and a root system $R$. We show that if  $\SA$ is of quantum characteristic $0$, the model category contains all $U_\SA$-modules that admit a Weyl filtration. If $\SA$ is in addition local, we study torsion phenomena in the model category. This leads to a construction of torsion free, or ``maximal'' objects in $\CX$. We show that these correspond to tilting modules for the quantum group associated with $\SA$ and $R$.

MSC classes (2020): 17B10, 17B37
\end{abstract}

\address{Department Mathematik, FAU Erlangen--N\"urnberg, Cauerstra\ss e 11, 91058 Erlangen}
\email{fiebig@math.fau.de}
\maketitle

\section{Introduction}
Let $U=U_\SA(R)$ be the quantum group associated with a (finite) root system $R$ and a $\DZ[v,v^{-1}]$-algebra $\SA$. We are interested in its category $\CO$ and specifically in the subcategory of $\CO$ that contains all objects that are finitely generated as $\SA$-modules. If $\SA$ is of quantum characteristic $0$, i.e.~if the quantum numbers $[n]$ do not vanish in $\SA$ for $n\ne 0$, then we can define for each dominant weight $\lambda$ the Weyl module $W(\lambda)$ in $\CO$.   By results of Andersen, Polo and Wen, $W(\lambda)$ is a free $\SA$-module of finite rank and its character is given by Weyl's character formula.

Recall that an object $T$ in $\CO$  is called a {\em tilting module} if $T$ and its contravariant dual $dT$ admit a Weyl filtration, i.e.~a finite filtration with subquotients being isomorphic to Weyl modules. In the case that $\SA=\SK$ is a field it is known that  tilting modules exist and that the indecomposable tilting modules are parametrized by their highest weight, which needs to be dominant (cf.~Appendix E in \cite{J}, where the modular case is treated). The existence of an indecomposable tilting module with highest (dominant) weight $\lambda$ is known in the case that $\SA$ is a principal ideal domain (cf.~Lemma E.19 in \cite{J}), and the uniqueness of these objects can be shown if $\SA$ is a complete discrete valuation ring (Proposition E.22 in \cite{J}). 

However, the case of $\SA=\SZ_{\fp}$, the localization of $\DZ[v,v^{-1}]$ at the kernel of the ring homomorphism $\DZ[v,v^{-1}]\to\DF_p$ that sends $v$ to $1$, is of particular importance, as $U_{\SZ_\fp}$ is a natural (quantum) deformation of the hyperalgebra of the connected, simply connected, split reductive algebraic group with root system $R$ over $\DF_p$.  In this case $\SA$ is a local ring of quantum characteristic $0$, but not a principal ideal domain, and the arguments in \cite{J} do not carry over. 
The main goal in this article is to fill this gap and prove that tilting modules exist over local rings of quantum characteristic $0$, and that the indecomposable tilting modules are parametrized (up to isomorphism) by their (dominant) highest weight. 

Our approach is very different from the approach of Jantzen, which is inspired by \cite{D}, which is in turn inspired by  \cite{R}. It involves a  category $\CX=\CX_\SA(R)$ that is defined as follows. Denote by $X$ the weight lattice of $R$, and fix a set $\Pi\subset R$ of simple roots. Then the category $\CX$ contains as objects $X$-graded $\SA$-modules $M=\bigoplus_{\lambda\in X}M_\lambda$ that are endowed with $\SA$-linear endomomorphisms $E_{\alpha,n}$ and $F_{\alpha,n}$ of degree $+n\alpha$ and $-n\alpha$, resp., for each  $\alpha\in\Pi$ and $n\in\DZ_{>0}$.  We state three rather simple axioms (X1, a boundedness condition on weights), (X2, a simple type $A_1$ commutation relation) and (X3, a replacement of the Serre relations) that ensure that $M$ carries a unique $U_\SA$-module structure such that the $X$-grading is the weight decomposition, and the $E_{\alpha,n}$ and $F_{\alpha,n}$ are the action maps of the divided powers of the Serre generators. This construction gives rise to a fully faithful  embedding of $\CX$ into the category $\CO$. This embedding is not an equivalence, but the image is big enough.  In the case that $\SA$ is of quantum characteristic $0$ we show that the objects in $\CX$ that are finitely generated over $\SA$ correspond bijectively to  the  objects in $\CO$ that admit a Weyl filtration. 

Working with the category $\CX$ has the advantage that one can construct objects locally, i.e.~weight space by weight space. Moreover, it provides a convenient framework to study torsion phenomena. In particular, to any object $M$ in $\CX$ and any weight $\mu$ we associate a certain triple  $M_{\{\mu\}}\subset M_{\lgl\mu\rgl}\subset  M_{\{\mu\},\max}$ of torsion free $\SA$-modules with  torsion quotients. We call an object $M$ in $\CX$ {\em minimal}, if $M_{\{\mu\}}=M_{\lgl\mu\rgl}$, and   {\em maximal}, if $M_{\lgl\mu\rgl}=M_{\{\mu\},\max}$. 
Then we show that minimal as well as  maximal objects exist that the indecomposables are in both cases parametrized by their highest weight, which can be any (not necessarily dominant) weight $\lambda$. So we obtain two families $S_{\min}(\lambda)$ and $S_{\max}(\lambda)$ in $\CX$ that we can consider, via the embedding above, as objects in $\CO$. In the case that $\lambda$ is dominant, the object $S_{\min}(\lambda)$ yields the Weyl module, and we show that the $S_{\max}(\lambda)$ are actually indecomposable tilting modules. This settles  the existence of tilting modules. The fact that two indecomposable tilting modules with the same highest weight are isomorphic is then a consequence of some properties of the maximal objects $S_{\max}(\lambda)$. 

{\bf Acknowledgement:} I would like to thank Henning Haahr Andersen, Jens Carsten Jantzen, and an anonymous referee for valuable comments on previous versions of this article.

\section{$X$-graded spaces with operators}\label{sec-Xgradop}
The main goal in this section is to define the category $\CX=\CX_\SA(R)$ for a finite root system $R$ and a unital Noetherian  domain $\SA$ that is a $\DZ[v,v^{-1}]$-algebra. 

\subsection{Quantum integers}\label{subsec-qint} Let $v$ be an indeterminate and set $\SZ:=\DZ[v,v^{-1}]$. For $n\in\DZ$ and $d>0$  we define the quantum integer 
$$
[n]_d:=\frac{v^{dn}-v^{-dn}}{v^d-v^{-d}}=\begin{cases}
0,&\text{ if $n=0$},\\
v^{d(n-1)}+v^{d(n-3)}+\dots+v^{d(-n+1)},&\text{ if $n> 0$}, \\
-v^{d(-n-1)}-v^{d(-n-3)}-\dots-v^{d(n+1)},&\text{ if $n<0$}.
\end{cases}
$$
The quantum factorials are given by   $[0]_d^!:=1$ and
$
[n]_d^!:=[1]_d\cdot[2]_d\cdots[n]_d
$ for $n\ge 1$. The quantum binomials are  $\qchoose{n}{0}_d:=1$ 
and 
$
\qchoose{n}{r}_d:=\frac{[n]_d\cdot[n-1]_d\cdots[n-r+1]_d}{[1]_d\cdot [2]_d\cdots [r]_d}$ for $n\in\DZ$ and $r\ge 1$.
Note that under the ring homomorphism $\SZ\to\DZ$ that sends $v$ to $1$, the quantum integer $[n]_d$ is sent to $n$ for all $n\in\DZ$, independently of $d$. Hence $[n]^!_d$ is sent to $n!$ and $\qchoose{n}{r}_d$ to $n\choose r$.

\subsection{Graded spaces with operators}\label{subsec-gradspac}
We fix a finite root system $ R$ in a real vector space $V$ and a basis $\Pi$ of $R$.  The coroot for $\alpha\in R$ is $\alpha^\vee\in V^\ast$, and the weight lattice is $X:=\{\lambda\in V\mid \lgl \lambda,\alpha^\vee\rgl\in\DZ\text{ for all $\alpha\in R$}\}$. We
denote by $\le$ the standard partial order on $X$, i.e.~$\mu\le\lambda$ if and only if $\lambda-\mu$ can be written as a sum of elements in $\Pi$.

\begin{definition} \begin{enumerate}
\item A subset $I$ of $X$ is called {\em closed} if $\mu\in I$ and $\mu\le\lambda$ imply $\lambda\in I$.
\item A subset $S$ of $X$ is called {\em  quasi-bounded} if for any $\mu\in X$ the set $\{\lambda\in S\mid \mu\le\lambda\}$ is finite.
\end{enumerate}
\end{definition}


 Now let  $\SA$ be a unital $\SZ$-algebra that is a Noetherian domain. We denote by $q\in\SA$ the image of $v$. Then $q$ is invertible in $\SA$. Conversely, giving a  $\SZ$-algebra structure on a unital ring $\SR$ is the same as specifying an invertible element $q$ of $\SR$. Let $I$ be a closed subset of $X$, and $M=\bigoplus_{\mu\in I}M_\mu$  an $I$-graded $\SA$-module. 
We say that $\mu$ is a {\em weight of $M$} if $M_\mu\ne\{0\}$. 
For any $\mu\in I$, $\alpha\in\Pi$, and $n>0$ we have $\mu+n\alpha\in I$. Let
\begin{align*}
F_{\mu,\alpha,n}&\colon M_{\mu+n\alpha}\to M_\mu,\\
E_{\mu,\alpha,n}&\colon M_{\mu}\to M_{\mu+n\alpha}
\end{align*}
be $\SA$-linear homomorphisms. 
It is convenient to set $E_{\mu,\alpha,0}=F_{\mu,\alpha,0}:=\id_{M_\mu}$.
In the following we often suppress the index ``$\mu$'' in the notation of the $E$- and $F$-maps if the source of the maps is clear from the context, but we sometimes also write $E_{\alpha,n}^M$ and $F_{\alpha,n}^M$ to specify the object $M$ on which these homomorphisms are defined.  

Now we list some conditions  on the above data. Denote by $A=(\lgl\alpha,\beta^\vee\rgl)_{\alpha,\beta\in\Pi}$ the Cartan matrix associated with the root system $R$. Then there exists a vector $d=(d_\alpha)_{\alpha\in\Pi}$ with entries in $\{1,2,3\}$ such that $(d_\alpha \lgl\alpha,\beta^\vee\rgl)_{\alpha,\beta\in\Pi}$ is symmetric and such that each irreducible component of $R$ contains some $\alpha$ with  $d_\alpha=1$. The first two conditions are as follows. 

\begin{enumerate}
\item[(X1)] The set of weights of $M$  is  quasi-bounded and each $M_\mu$ is finitely generated as an $\SA$-module. 
\item[(X2)] For all $\mu\in I$, $\alpha,\beta\in\Pi$, $m,n>0$, and $v\in M_{\mu+n\beta}$,
$$
E_{\alpha,m}F_{\beta,n}(v)=
\begin{cases}
F_{\beta,n}E_{\alpha,m}(v),&\text{ if $\alpha\ne\beta$,}\\
\sum_{r=0}^{ \min(m,n)} \qchoose{\lgl\mu,\alpha^\vee\rgl+m+n}{r}_{d_\alpha}F_{\alpha,n-r}E_{\alpha,m-r}(v),&\text{ if $\alpha=\beta$}.
\end{cases}
$$
\end{enumerate}
(The cautious reader may want to have a look at Equation (a2) in Section 6.5 of \cite{L90} to get an idea of where the second equation comes from.)

\subsection{Torsion subquotients}
In order to formulate the third condition, we need some definitions. Suppose that $M$ satisfies (X1). Then for any $\mu\in I$ we can define
$$
M_{\delta\mu}:=\bigoplus_{\alpha\in\Pi, n>0} M_{\mu+n\alpha}.
$$
Note that, since the set of weights is quasi-bounded, only finitely many of the direct summands are non-zero. 
Let
\begin{align*}
E_\mu&\colon M_\mu\to M_{\delta\mu},\\
F_\mu&\colon M_{\delta\mu}\to M_\mu
\end{align*}
be the column and the row vector with entries $E_{\mu,\alpha,n}$ and $F_{\mu,\alpha,n}$, resp. We sometimes write $E_\mu^M$ and $F_\mu^M$ in order to  specify the object $M$ on which $E_\mu$ and $F_\mu$ act. Set
\begin{align*}
M_{\{\mu\}}&:=E_\mu(\im F_\mu),\\
M_{\lgl\mu\rgl}&:=E_\mu(M_\mu).
\end{align*}
So we have inclusions $M_{\{\mu\}}\subset M_{\lgl\mu\rgl}\subset M_{\delta\mu}$.

\begin{enumerate}
\item[(X3)] For all $\mu\in I$ the following holds:
\begin{enumerate}
\item The restriction of $E_\mu\colon M_\mu\to M_{\delta\mu}$ to $\im F_\mu\subset M_{\mu}$ is injective and hence induces an isomorphism $\im F_\mu\xrightarrow{\sim} M_{\{\mu\}}$.
\item  The quotient $M_{\lgl\mu\rgl}/M_{\{\mu\}}$ is a torsion $\SA$-module. 
\item The quotient $M_\mu/\im F_\mu$ is a free $\SA$-module.
\end{enumerate} 
\end{enumerate}

Here is our first, rather easy, result.

\begin{lemma} \label{lemma-torfree} Suppose that our data satisfies (X1) and (X3). Then $M$ is a torsion free $\SA$-module. In particular, the spaces $M_{\delta\mu}$, $M_{\{\mu\}}$ and $M_{\lgl\mu\rgl}$  are torsion free $\SA$-modules for all $\mu\in I$.
\end{lemma}

\begin{proof} If $M$ is not torsion free, then assumption (X1) implies that there is a maximal weight $\mu$ of $M$ such that $M_\mu$ is not torsion free. By the maximality of $\mu$,  the module $M_{\delta\mu}$ is torsion free. Hence so is its submodule $M_{\{\mu\}}$. From (X3a) it follows that $\im F_\mu$ is torsion free. By (X3c) the module  $M_\mu/\im F_\mu$ is free,  so $M_\mu$ must be torsion free and we have a contradiction.  
\end{proof}
The following results might shed some light on the assumption (X3). 

\begin{lemma} \label{lemma-minimal} Suppose that the assumption (X3a) holds and let $\mu$ be an element in $I$. Then  the following are equivalent.
\begin{enumerate}
\item $M_\mu=\ker E_\mu\oplus\im F_\mu$.
\item $M_{\{\mu\}}=M_{\lgl\mu\rgl}$.  
\end{enumerate}
\end{lemma}
\begin{proof} If (1) holds, then $M_{\lgl\mu\rgl}=E_\mu(M_\mu)=E_\mu(\im F_\mu)=M_{\{\mu\}}$. Suppose that (2) holds, so $E_\mu(M_\mu)=E_\mu(\im F_\mu)$. For each $m\in M_\mu$ there exists then an element $\tilde m\in \im F_\mu$ such that $E_\mu(m)=E_\mu(\tilde m)$, hence $m-\tilde m\in\ker E_\mu$. So $M_\mu=\ker E_\mu+\im F_\mu$. But condition (X3a) reads $\ker E_\mu\cap\im F_\mu=\{0\}$. Hence $M_\mu=\ker E_\mu\oplus\im F_\mu$. 
\end{proof}

Denote by $\SK$ the quotient field of $\SA$. For an $\SA$-module $N$  let $N_\SK:=N\otimes_\SA\SK$ be the associated $\SK$-module.

\begin{lemma}\label{lemma-X3} Suppose that $M=\bigoplus_{\mu\in I} M_\mu$ satisfies condition (X1). Then condition (X3) is equivalent to the following set of conditions.
\begin{enumerate}
\item $M$ is a torsion free $\SA$-module.
\item  For all $\mu \in I$ we have $(M_\mu)_\SK=(\ker E_\mu)_\SK\oplus (\im F_\mu)_\SK$.
\item  Condition (X3c) holds: $M_\mu/\im F_\mu$ is a free $\SA$-module for all $\mu\in I$.
\end{enumerate}
In particular, if $\SA=\SK$ is a field, then the conditions (X3abc) simplify to $M_\mu=\ker E_\mu\oplus \im F_\mu$ for all $\mu\in I$. 
\end{lemma}
\begin{proof} Suppose that (X3) is satisfied. We have already shown in Lemma \ref{lemma-torfree} that (X1) and (X3) imply that $M$ is torsion free as an $\SA$-module.  Moreover, (X3a) says that $\ker E_\mu\cap \im F_\mu=\{0\}$. Now let $m\in M_\mu$. Then, by (X3b), there exists an element $\xi\in\SA$, $\xi\ne 0$ and $m^\prime\in \im F_\mu$ such that $\xi E_\mu(m)=E_\mu(m^\prime)$. So $\xi m-m^\prime$ is contained in the kernel of $E_\mu$ and we deduce $(M_\mu)_\SK=(\ker E_\mu)_\SK+(\im F_\mu)_\SK$. The last two results say that $(M_\mu)_\SK=(\ker E_\mu)_\SK\oplus (\im F_\mu)_\SK$. Hence (1), (2) and (3) are satisfied. 

Now assume that (1), (2) and (3) hold. As $M$ is torsion free we can view it as a subspace in $M_\SK$. Hence $(M_\mu)_\SK=(\ker E_\mu)_\SK\oplus (\im F_\mu)_\SK$ implies that $E_\mu|_{\im F_\mu}$ is injective, i.e.~(X3a). It also implies that 
$E_\mu(\im F_\mu)_\SK=E_\mu(M_\mu)_\SK$, i.e.~$(M_{\{\mu\}})_\SK=(M_{\lgl\mu\rgl})_\SK$. Hence  the cokernel of the inclusion $M_{\{\mu\}}\subset M_{\lgl\mu\rgl}$  is a torsion module, so (X3b) holds. That (X3c) holds is the assumption (3). 
\end{proof}

\subsection{The category $\CX$} 
Let $I$ be a closed subset of $X$.
\begin{definition} The category $\CX_{I,\SA}$ is defined as follows. Objects are $I$-graded $\SA$-modules $M=\bigoplus_{\mu\in I}M_\mu$ endowed with $\SA$-linear homomorphisms $F_{\mu,\alpha,n}\colon M_{\mu+n\alpha}\to M_\mu$ and $E_{\mu,\alpha,n}\colon M_\mu\to M_{\mu+n\alpha}$ for all $\mu\in I$, $\alpha\in\Pi$ and $n>0$, such that conditions (X1), (X2) and (X3) are satisfied. A morphism $f\colon M\to N$  in $\CX_{I,\SA}$ is a collection of $\SA$-linear homomorphisms $f_\mu\colon M_\mu\to N_\mu$ for all $\mu\in I$ such that the diagrams

\centerline{
\xymatrix{
M_{\mu+n\alpha}\ar[r]^{f_{\mu+n\alpha}}\ar[d]_{F^M_{\alpha,n}}&N_{\mu+n\alpha}\ar[d]^{F^N_{\alpha,n}}\\
M_\mu\ar[r]^{f_\mu}&N_\mu
}
\quad\quad
\xymatrix{
M_{\mu+n\alpha}\ar[r]^{f_{\mu+n\alpha}}&N_{\mu+n\alpha}\\
M_{\mu}\ar[u]^{E^M_{\alpha,n}}\ar[r]^{f_{\mu}}&N_{\mu}\ar[u]_{E^N_{\alpha,n}}
}
}
\noindent
commute for all $\mu\in I$, $\alpha\in\Pi$ and $n>0$.
\end{definition}
If the ground ring is determined from the context, we write $\CX_I$ instead of $\CX_{I,\SA}$. We also write $\CX$ or $\CX_\SA$ for the ``global''  category $\CX_{X,\SA}$. 
 
If $M$ and $N$ are objects in $\CX_I$ and $f=\{f_\mu\colon M_\mu\to N_\mu\}_{\mu\in I}$ is a collection of homomorphisms, then we denote  by $f_{\delta\mu}\colon M_{\delta\mu}\to N_{\delta\mu}$ the diagonal matrix with entries $f_{\mu+n\alpha}$. Then $f$ is a morphism in $\CX_I$ from $M$ to $N$ if and only if for all $\mu\in I$ the diagrams

\centerline{
\xymatrix{
M_{\delta\mu}\ar[d]_{F^M_\mu}\ar[r]^{f_{\delta\mu}}&N_{\delta\mu}\ar[d]^{F^N_\mu}\\
M_\mu\ar[r]^{f_\mu}& N_\mu
}
\quad\quad
\xymatrix{
M_{\delta\mu}\ar[r]^{f_{\delta\mu}}&N_{\delta\mu}\\
M_\mu\ar[u]^{E^M_\mu}\ar[r]^{f_\mu}& N_\mu\ar[u]_{E^N_\mu}
}
}
\noindent commute.

\subsection{Base change}\label{sec-basech} We now want to understand whether the conditions that define the category $\CX$ are stable under base change. 
So let $\SA\to\SB$  be a homomorphism of unital $\SZ$-algebras that are Noetherian  domains. Let $M$ be an object in $\CX_{I,\SA}$. We define $M_{\SB}=\bigoplus_{\mu\in I} M_{\SB\mu}$ by setting (as before) $M_{\SB\mu}:=M_\mu\otimes_\SA\SB$. For $\mu\in X$, $\alpha\in\Pi$ and $n>0$ we have induced homomorphisms $E^{M_{\SB}}_{\mu,\alpha,n}=E_{\mu,\alpha,n}\otimes\id_{\SB}\colon M_{\SB\mu}\to M_{\SB\mu+n\alpha}$ and $F^{M_{\SB}}_{\mu,\alpha,n}=F_{\mu,\alpha,n}\otimes\id_{\SB}\colon M_{\SB\mu+n\alpha}\to M_{\SB\mu}$.

\begin{lemma} \label{lemma-basechange} Suppose that $\SA\to\SB$ is a flat homomorphism. Then the  object $M_\SB$ is contained in $\CX_{I,\SB}$.
\end{lemma}
\begin{proof} It is clear that the properties (X1) and (X2) are stable under arbitrary base change. Moreover,  $M_{\SB\lgl\mu\rgl}$, $M_{\SB\{\mu\}}$ and  $\im F_{\mu}^{M_\SB}$  are obtained from $M_{\lgl\mu\rgl}$, $M_{\{\mu\}}$ and $\im F_\mu^M$ by flat base change for all $\mu\in I$ by right exactness of base change. Hence (X3b) and (X3c) also hold for $M_\SB$. Again by the flatness condition, the homomorphism $E_\mu|_{\im F_\mu}$ remains injective after base change. Hence property  (X3a) also holds.  \end{proof}

\section{Extending morphisms}\label{sec-extmor}
We retain the notations of the previous section. 
Let $I^\prime\subset I$ be closed subsets of $X$ and let $M$ be an object in $\CX_I$. We define $M_{I^\prime}:=\bigoplus_{\mu\in I^\prime} M_\mu$ and endow it with the homomorphisms $E_{\mu,\alpha,n}$ and $F_{\mu,\alpha,n}$ for all $\mu\in I^\prime$. Then one easily checks that the properties (X1), (X2) and (X3) are preserved, so this defines an object $M_{I^\prime}$  in $\CX_{I^\prime}$. For a morphism $f\colon M\to N$ in $\CX_I$ we obtain a morphism $f_{I^\prime}\colon M_{I^\prime}\to N_{I^\prime}$ by restriction, and this yields a functor
$$
(\cdot)_{I^\prime}\colon \CX_I\to \CX_{I^\prime}
$$
that we call the {\em restriction functor}.

\subsection{Extensions of morphisms}

The following proposition is a cornerstone of the approach outlined in this article. Its proof is not difficult, but lengthy.  

\begin{proposition} \label{prop-mainext}  Let $I^\prime$ be a closed subset of $X$ and suppose that $\mu\not\in I^\prime$ is such that $I:=I^\prime\cup\{\mu\}$ is also closed. Let $M$ and $N$ be objects in $\CX_I$, and let $f^\prime\colon M_{I^\prime}\to N_{I^\prime}$ be a morphism in $\CX_{I^\prime}$.
\begin{enumerate}
\item There exists a unique $\SA$-linear homomorphism $\tilde f_\mu\colon \im F^M_\mu\to N_\mu$ such that the diagrams 

\centerline{
\xymatrix{
M_{\delta\mu}\ar[d]_{F_\mu^M}\ar[r]^{f^\prime_{\delta\mu}}&N_{\delta\mu}\ar[d]^{F_\mu^N}\\
\im F_\mu^M\ar[r]^{\tf_\mu}& N_\mu
}
\quad\quad
\xymatrix{
M_{\delta\mu}\ar[r]^{f^\prime_{\delta\mu}}&N_{\delta\mu}\\
\im F_\mu^M\ar[u]^{E_\mu^M|_{\im F_\mu^M}}\ar[r]^{\tf_\mu}& N_\mu\ar[u]_{E_\mu^N}
}
}
\noindent commute. In particular, $f^\prime_{\delta\mu}$ maps $M_{\{\mu\}}$ into $N_{\{\mu\}}$.
\item The following are equivalent.
\begin{enumerate}
\item  There exists a morphism $f\colon M\to N$ in $\CX_I$ such that $f_{I^\prime}=f^\prime$.
\item The homomorphism $f^\prime_{\delta\mu}\colon M_{\delta\mu}\to N_{\delta\mu}$ maps $M_{\lgl\mu\rgl}$ into $N_{\lgl\mu\rgl}$.
\end{enumerate}
\end{enumerate}
\end{proposition}
\begin{proof} First we prove part (1). 
Set $\hM_\mu:=\bigoplus_{\beta\in\Pi,n>0} M_{\mu+n\beta}$ and denote by $\hF_{\beta,n}\colon M_{\mu+n\beta}\to \hM_\mu$ the embedding of the corresponding direct summand. Define $\hF_\mu\colon M_{\delta\mu}\to \hM_\mu$ as the row vector with entries $\hF_{\beta,n}$\footnote{The author is aware of the fact that this looks rather silly. There is a tautological identification $M_{\delta\mu}=\widehat M_\mu$ that identifies $\hF_\mu$ with the identity. However, $M_{\delta\mu}$ and $\widehat M_\mu$ will play very different roles in the following.}. For $\alpha\in\Pi$, $m>0$  define an $\SA$-linear map $\widehat E_{\alpha,m}\colon \hM_\mu\to M_{\mu+m\alpha}$ by additive extension of the following formulas. For $\beta\in\Pi$, $n>0$ and $v\in M_{\mu+n\beta}$ set
$$
\widehat E_{\alpha,m}\widehat F_{\beta,n}(v):=
\begin{cases}
F_{\beta,n}E_{\alpha,m}(v),&\text{ if $\alpha\ne\beta$},\\
\sum_{0\le r\le \min(m,n)} \qchoose{\lgl\mu,\alpha^\vee\rgl+n+m}{ r}_{d_\alpha} F_{\alpha,n-r}E_{\alpha,m-r}(v), &\text{ if $\alpha=\beta$}.
\end{cases}
$$
Let $\hE_\mu\colon \hM_\mu\to M_{\delta\mu}$ be the column vector with entries $\hE_{\alpha,m}$. 

Now define $\phi\colon \hM_\mu\to M_\mu$ as the row vector with entries $F_{\beta,n}$.  Obviously, the diagram

\centerline{
\xymatrix{
&M_{\delta\mu}\ar[dl]_{\hF_\mu}\ar[dr]^{F_\mu}&\\
\hM_\mu\ar[rr]^\phi&& M_\mu
}
}
\noindent
commutes. As the $\hE_{\alpha,m}$- and $\hF_{\beta,n}$-maps satisfy the same commutation relations as the $E_{\alpha,m}$- and $F_{\beta,n}$-maps by (X2), and as $\hF_{\mu}$ is surjective, also the diagram

\centerline{
\xymatrix{
&M_{\delta\mu}&\\
\hM_\mu\ar[rr]^\phi\ar[ur]^{\hE_\mu}&& M_\mu\ar[ul]_{E_\mu}
}
}
\noindent
commutes. 
 As $\hF_\mu$ is surjective, we have $\im\phi=\im F_\mu$. As $E_\mu$ is injective when restricted to $\im F_\mu$, we deduce that $\ker\phi=\ker \hE_\mu$, hence $\phi$ induces an isomorphism $\hM_\mu/\ker \hE_\mu\cong\im F_\mu$.

Now let  $\hf_\mu\colon \hM_\mu\to N_\mu$ be the row vector with entries $F^N_{\beta,n}\circ f^\prime_{\mu+n\beta}\colon M_{\mu+n\beta}\to N_{\mu+n\beta}\to N_\mu$. Then   the diagram
 
\centerline{
\xymatrix{
M_{\delta\mu}\ar[d]_{\hF_\mu}\ar[r]^{f^\prime_{\delta\mu}}&N_{\delta\mu}\ar[d]^{F^N_\mu}\\
\hM_\mu\ar[r]^{\hf_\mu}& N_\mu
}
}
\noindent commutes. By the same arguments as above, also the diagram

\centerline{ 
\xymatrix{
M_{\delta\mu}\ar[r]^{f^\prime_{\delta\mu}}&N_{\delta\mu}\\
\hM_\mu\ar[u]^{\hE_\mu}\ar[r]^{\hf_\mu}& N_\mu\ar[u]_{E^N_\mu}
}
}
\noindent commutes.
As $\hF_\mu$ is surjective, the image of $\hf_\mu$ is contained in $\im F_\mu^N\subset N_\mu$. As $E_\mu^N$ is injective when restricted to $\im F_\mu^N$, we deduce that $\hf_\mu$ factors over the kernel of $\hE_\mu$. But, as we have seen above, this is the kernel of $\phi$. We hence obtain an induced homomorphism $\tf_\mu\colon \im F^M_\mu\cong\hM_\mu/\ker\phi\to N_\mu$ such that the diagrams

\centerline{
\xymatrix{
M_{\delta\mu}\ar[d]_{F_\mu^M}\ar[r]^{f^\prime_{\delta\mu}}&N_{\delta\mu}\ar[d]^{F^N_\mu}\\
\im F_\mu^M\ar[r]^{\tf_\mu}& N_\mu
}
\quad\quad
\xymatrix{
M_{\delta\mu}\ar[r]^{f^\prime_{\delta\mu}}&N_{\delta\mu}\\
\im F_\mu^M\ar[u]^{E^M_\mu}\ar[r]^{\tf_\mu}& N_\mu\ar[u]_{E^N_\mu}
}
}
\noindent commute. This shows the existence part of (1). The uniqueness is clear, as $F_\mu^M \colon M_{\delta\mu}\to \im F_\mu^M$ is surjective. 

Now we show part (2). Assume that property (a) holds, i.e.~there exists a homomorphism $f\colon M\to N$ that restricts to $f^\prime$. Then the diagram
 
\centerline{ 
\xymatrix{
M_{\delta\mu}\ar[rr]^{f^\prime_{\delta\mu}=f_{\delta\mu}}&&N_{\delta\mu}\\
M_\mu\ar[u]^{E^M_\mu}\ar[rr]^{f_\mu}&& N_\mu\ar[u]_{E^N_\mu}
}
}
\noindent commutes and hence $f^\prime_{\delta\mu}$ maps $M_{\lgl\mu\rgl}=E^M_\mu(M_\mu)$ into $N_{\lgl\mu\rgl}=E^N_\mu(N_\mu)$, so property (b) holds.

Now assume property (b) holds. We now need to construct an $\SA$-linear map $f_\mu\colon M_\mu\to N_\mu$ such that the diagrams

$$
%
\leqno{(\ast)}
\begin{gathered}\quad\quad
\hbox{
\xymatrix{
M_{\delta\mu}\ar[d]_{F^M_\mu}\ar[r]^{f^\prime_{\delta\mu}}&N_{\delta\mu}\ar[d]^{F_\mu^N}\\
M_\mu\ar[r]^{f_\mu}& N_\mu
}
\quad\quad
\xymatrix{
M_{\delta\mu}\ar[r]^{f^\prime_{\delta\mu}}&N_{\delta\mu}\\
M_\mu\ar[u]^{E_\mu^M}\ar[r]^{f_\mu}& N_\mu\ar[u]_{E_\mu^N}
}
}
\end{gathered}
$$
commute. By part (1), there exists a homomorphism $\tf_\mu\colon \im F_\mu^M\to N_\mu$ such that the diagrams

$$
\begin{gathered}
\xymatrix{
M_{\delta\mu}\ar[d]_{F_\mu^M}\ar[r]^{f^\prime_{\delta\mu}}&N_{\delta\mu}\ar[d]^{F_\mu^N}\\
\im F_\mu^M\ar[r]^{\tf_\mu}& N_\mu
}
\quad\quad
\xymatrix{
M_{\delta\mu}\ar[r]^{f^\prime_{\delta\mu}}&N_{\delta\mu}\\
\im F_\mu^M\ar[u]^{E_\mu^M}\ar[r]^{\tf_\mu}& N_\mu\ar[u]_{E_\mu^N}
}
\end{gathered}
$$
commute. By assumption (X3c), the quotient $M_\mu/\im F^M_\mu$ is a free $\SA$-module. We can hence fix a decomposition $M_\mu=\im F^M_\mu\oplus D$ with a free $\SA$-module $D$. We now construct a homomorphism $\hf_\mu\colon D\to N_\mu$ in such a way that  $f_\mu:=(\tf_\mu,\hf_\mu)$ serves our purpose. Note that no matter how we define $\hf_\mu$, we will always have $f_\mu\circ F_\mu^M=F_\mu^N\circ f_{\delta\mu}^\prime$ (cf.~the left diagram of $(\ast)$). So the only property that $\hf_\mu$ has to satisfy is that the diagram

\centerline{
\xymatrix{
M_{\delta\mu}\ar[r]^{f^\prime_{\delta\mu}}&N_{\delta\mu}\\
D\ar[u]^{E_\mu^M|_D}\ar[r]^{\hf_\mu}& N_\mu\ar[u]_{E_\mu^N}
}
}
\noindent commutes. Since we assume that $f^\prime_{\delta\mu}(E_\mu^M(M_{\mu}))$ is contained in the image of $E^N_\mu\colon N_\mu\to N_{\delta\mu}$, this also holds for $f^\prime_{\delta\mu}(E^M_\mu(D))$.  As $D$ is free, it is projective as an $\SA$-module. So $\hf_\mu$ indeed exists.
\end{proof}
\begin{remark} In part (2) of the lemma above, the extension $f$ of $f^\prime$ is in general not unique. In the notation of the proof of part (2), the $\SA$-linear homomorphism $\hf_\mu$ is in general not unique, nor is the decomposition $M_\mu=\im F_\mu^M\oplus D$.
\end{remark}

\subsection{Minimal and  maximal  objects}
 Let $I$ be a closed subset of $X$ and let  $M$ be an object in $\CX_I$. Let $\mu\in I$. Define
$$
M_{\{\mu\},\max}:=\{m\in M_{\delta\mu}\mid \xi m \in M_{\{\mu\}}\text{ for some $\xi\in\SA$, $\xi\ne 0$}\}.
$$
So this is the preimage of the torsion part of $M_{\delta\mu}/M_{\{\mu\}}$ under the quotient map. 
Suppose that $N$ is a submodule of $M_{\delta\mu}$ that contains $M_{\{\mu\}}$.
Then $N/M_{\{\mu\}}$ is a torsion module if and only if $N\subset M_{\{\mu\},\max}$. In particular, condition (X3b) now reads 
$
M_{\lgl\mu\rgl}\subset M_{\{\mu\},\max}
$.
So we have inclusions 
 $$
M_{\{\mu\}}\subset M_{\lgl\mu\rgl}\subset M_{\{\mu\},\max}\subset M_{\delta\mu}.
$$
We now give the two extreme cases a name.
 \begin{definition} $M$ is called 
\begin{enumerate}
\item 
 {\em minimal}, if  for all $\mu\in I$ we have $M_{\{\mu\}}=M_{\lgl\mu\rgl}$.
 \item 
 {\em maximal}, if  for all $\mu\in I$ we have $M_{\lgl\mu\rgl}=M_{\{\mu\},\max}$.
 \end{enumerate}
\end{definition}

In the two extreme cases we can extend morphisms according to the following result.
\begin{lemma}\label{lemma-minimalextmor} Suppose that  $I^\prime\subset X$ is closed and that $\mu\not\in I^\prime$ is such that $I:=I^\prime\cup\{\mu\}$ is closed in $X$ as well. 
 Let $M$ and $N$ be objects in $\CX_I$ and suppose that $M_{\{\mu\}}=M_{\lgl\mu\rgl}$ or $N_{\lgl\mu\rgl}=N_{\{\mu\},\max}$. Then the functorial map
$$
\Hom_{\CX_I}(M,N)\to \Hom_{\CX_{I^\prime}}(M_{I^\prime},N_{I^\prime})
$$
is surjective. 
\end{lemma}

\begin{proof} Let $f^\prime\colon M_{I^\prime}\to N_{I^\prime}$ be a morphism in $\CX_{I^\prime}$.  By Proposition \ref{prop-mainext}, there exists a (unique) $\tf_\mu\colon \im F^M_\mu\to N_\mu$ such that the diagrams 

\centerline{
\xymatrix{
M_{\delta\mu}\ar[d]_{F_\mu^M}\ar[r]^{f^\prime_{\delta\mu}}&N_{\delta\mu}\ar[d]^{F_\mu^N}\\
\im F_\mu^M\ar[r]^{\tf_\mu}& N_\mu
}
\quad\quad
\xymatrix{
M_{\delta\mu}\ar[r]^{f^\prime_{\delta\mu}}&N_{\delta\mu}\\
\im F_\mu^M\ar[u]^{E_\mu^M}\ar[r]^{\tf_\mu}& N_\mu\ar[u]_{E_\mu^N}
}
}
\noindent commute. This implies that  $f^\prime_{\delta\mu}$ maps $M_{\{\mu\}}$  into $N_{\{\mu\}}$ and hence $M_{\{\mu\},\max}$  into $N_{\{\mu\},\max}$. Either of the two  assumptions in the statement of the lemma implies that $f^\prime_{\delta\mu}$ maps  $M_{\lgl\mu\rgl}$ into $N_{\lgl\mu\rgl}$, so the condition (2b)  in Proposition \ref{prop-mainext} is satisfied. Hence there exists an extension $f\colon M\to N$ of $f^\prime$. 
\end{proof}

\section{The minimal extension functor}\label{sec-ExtObj}
Again we retain all notations. In the last section we studied assumptions that ensure that morphisms in $\CX$ can be extended. In this section we want to extend objects, i.e.~we construct for each pair $I^\prime\subset I$ of closed subsets of $X$ a functor $\mathsf{E}_{I^\prime}^I\colon \CX_{I^\prime}\to \CX_I$ that is fully faithful and left adjoint to the restriction functor $(\cdot)_{I^\prime}\colon \CX_I\to\CX_{I^\prime}$. We call it the {\em minimal extension functor}.

\subsection{The minimal extension object}\label{subsec-minext}  
In this section we construct a {\em minimal extension} in $\CX_I$ for any object in $\CX_{I^\prime}$. In the next section, we show that the construction is actually functorial.
 
\begin{proposition}\label{prop-minimalextobj} Let $I^\prime\subset I$ be a pair of closed subsets of $X$. Let $M^\prime$ be an object in $\CX_{I^\prime}$. 
\begin{enumerate}
\item There exists an up to isomorphism unique object $M$ in $\CX_I$ with the following properties.
\begin{enumerate}
\item The object $M$ restricts to $M^\prime$, i.e.~$M_{I^\prime}\cong M^\prime$. 
\item For all objects $N$ in $\CX_I$  the functorial homomorphism
$$
\Hom_{\CX_I}(M,N)\to \Hom_{\CX_{I^{\prime}}}(M_{I^{\prime}},N_{I^{\prime}})
$$
is an isomorphism. 
\end{enumerate}
\item For the object $M$ characterized in part (1) we have $M_\mu=\im F_\mu$ and hence $M_{\{\mu\}}=M_{\lgl\mu\rgl}$ for all $\mu\in I\setminus I^\prime$, and $M_\mu\ne 0$ implies that there exists a weight $\lambda$ of $M^\prime$ with $\mu\le\lambda$. 
\end{enumerate}
\end{proposition}
\begin{proof} Note that the uniqueness statement in (1) follows directly  from properties (1a) and  (1b).  So, in order to prove (1), we   only need to show the existence of an object $M$ satisfying (1a) and (1b).  For this we  give an explicit construction. Note that we can construct $M$ in steps, i.e.~if we have closed subsets $I^\prime\subset I^{\prime\prime}\subset I$ and we construct an object $M^{\prime\prime}$ in $\CX_{I^{\prime\prime}}$ satisfying properties (1a), (1b) and (2) with respect to $M^\prime$, and then we construct an object $M$ in $\CX_I$ that satisfies properties (1a), (1b) and (2) with respect to $M^{\prime\prime}$, then $M$ satisfies (1a), (1b) and (2) with respect to $M^\prime$.

So in a first step we set $M_\nu=M^\prime_\nu$, $E_{\nu,\alpha,n}^{M}=E_{\nu,\alpha,n}^{M^\prime}$ and $F_{\nu,\alpha,n}^{M}=F_{\nu,\alpha,n}^{M^\prime}$ for all $\nu\in I^\prime$, $\alpha\in\Pi$ and $n>0$ in order to make sure that (1a) is satisfied.  
Now let $S$ be the set of all   $\mu\in I\setminus I^\prime$ that have the property that  there exists no $\lambda$ with $\mu<\lambda$ such that $M^\prime_\lambda\ne\{0\}$. For all $\mu\in S$ we set $M_\mu=0$, $E^M_{\mu,\alpha,n}=0$ and $F_{\mu,\alpha,n}^M=0$. Note that $I^\prime\cup S$ is closed. Now we have constructed an object $M$ in $\CX_{I^\prime\cup S}$ that obviously satisfies (1a). For all $\nu\in S\setminus I^\prime$ we have $M_\nu=0$, hence $0=M_{\{\nu\}}=M_{\lgl\nu\rgl}$, hence (2) is satisfied as well. As $\im F_\nu=M_\nu=0$ for all these $\nu$, property (1b) follows from part (1) in Proposition \ref{prop-mainext}.

Hence we have extended $M^\prime$ to an object in $\CX_{I^\prime\cup S}$ in such a way that properties (1a), (1b) and (2) are satisfied. 
Now note that the set $I\setminus (I^\prime\cup S)$ is quasi-bounded because the set of weights of $M^\prime$ is. We can now proceed by induction and assume that $I\setminus (I^\prime\cup S)=\{\mu\}$ for a single element $\mu\in I$. 
Then we can already define $M_{\delta\mu}:=\bigoplus_{\beta\in\Pi, n>0} M_{\mu+n\beta}$. For the construction of $M_\mu$ and $E_{\mu,\alpha,n}$ and $F_{\mu,\beta,n}$ we follow ideas that were already used  in the proof of Proposition \ref{prop-mainext}. So in a first step we  set $\hM_\mu:=\bigoplus_{\beta\in\Pi,n>0} M_{\mu+n\beta}$ and denote by $\hF_{\beta,n}\colon M_{\mu+n\beta}\to \hM_{\mu}$ the canonical injection of a direct summand. We let $\hF_\mu\colon M_{\delta\mu}\to\hM_\mu$ be the row vector with entries $\hF_{\beta,n}$ (again, $\hF_\mu$ is the identity).  For $\alpha\in\Pi$, $m>0$  define an $\SA$-linear map $\widehat E_{\alpha,m}\colon \hM_\mu\to M_{\mu+m\alpha}$ by additive extension of the following formulas. For $\beta\in\Pi$, $n>0$ and $v\in M_{\mu+n\beta}$ set
$$
\widehat E_{\alpha,m}\widehat F_{\beta,n}(v):=
\begin{cases}
F_{\beta,n}E_{\alpha,m}(v),&\text{ if $\alpha\ne\beta$},\\
\sum_{r=0}^{\min(m,n)} \qchoose{\lgl\mu,\alpha^\vee\rgl+m+n}{ r}_{d_\alpha} F_{\alpha,n-r}E_{\alpha,m-r}(v), &\text{ if $\alpha=\beta$}.
\end{cases}
$$
We denote by $\hE_\mu\colon \hM_\mu\to M_{\delta\mu}$ the column vector with entries $\hE_{\alpha,m}$. 
Now define $M_\mu:=\hM_\mu/\ker \hE_\mu$, and denote by $E_\mu\colon M_\mu\to M_{\delta\mu}$ and $F_\mu\colon M_{\delta\mu}\to M_\mu$ the homomorphisms induced by $\hE_\mu$ and $\hF_\mu$, resp. Note that $F_\mu$ is surjective since $\hF_\mu$ is. Denote by $E_{\mu,\alpha,n}\colon M_\mu\to M_{\mu+n\alpha}$ and by $F_{\mu,\alpha,n}\colon M_{\mu+n\alpha}\to M_\mu$ the entries of the row vector $E_\mu$ and the column vector $F_\mu$, resp. 

We claim that the above data  yields  an object in $\CX_I$. Clearly, property (X1) is satisfied. Also, the commutation relations between the $E$- and $F$-maps follow from the  relations between the respective maps on $M^\prime$ and the construction of  $E_\mu$ and $F_\mu$. Hence (X2) is satisfied as well. The properties (X3) are satisfied for all weights $\nu$ with $\nu\ne\mu$, as they are satisfied for $M^\prime$. For the weight $\mu$, however, we have  $\ker E_\mu=\{0\}$, hence (X3a) is satisfied, and $M_\mu=\im F_\mu$, so $M_{\lgl\mu\rgl}=M_{\{\mu\}}$, which imply (X3b) and (X3c). 

It remains to show that the object $M$ satisfies the properties (1a) and (1b). Part (1a) is clear from the construction. Part (1b) follows from $M_\mu=\im F_\mu$ and part (1) of Proposition  \ref{prop-mainext}.  Hence (1) is proven. Clearly $M_\mu\ne 0$ implies $M^\prime_{\delta\mu}\ne 0$, hence there is some weight $\lambda$ of $M^\prime$ with $\mu\le\lambda$. Since $M_{\mu}=\im F_\mu$ we have $M_{\{\mu\}}=M_{\lgl\mu\rgl}$, hence (2). 
\end{proof}

\subsection{A first example}\label{ex-first} 
Suppose that $R=\{\pm\alpha\}$ is the root system of type $A_1$ with $\Pi=\{\alpha\}$. We identify the weight lattice $X$ with $\DZ$ by sending $\alpha\in X$ to $2\in\DZ$. Let $\lambda\in\DZ$ and set $I^\prime=\{\lambda,\lambda+2,\lambda+4,\dots\}$, $\mu=\lambda-2$ and $I=I^\prime\cup\{\mu\}$. Define $M^\prime:=\bigoplus_{\nu\in I^\prime}M^\prime_\nu$ with $M^\prime_\lambda=\SA$ and $M^\prime_\nu=0$ for $\nu\ne\lambda$. All relevant $E$- and $F$-maps are zero, of course. In the notation of the proof of Proposition \ref{prop-minimalextobj} we have $\widehat M_{\mu}=M_{\delta_\mu}=\SA$ and $\widehat F_\mu\colon M_{\delta\mu}\to \widehat M_\mu$ is the identity. In the next step we define the map $\widehat E_\mu\colon \widehat M_{\mu}\to M_{\delta\mu}$. The commutation relations force this map to be multiplication with the image of $\qchoose{\mu+2}{1}=[\lambda]$ in $\SA$. So if  $[\lambda]$ does not vanish in $\SA$, the map $\widehat E_\mu$  is injective, so $M_\mu=\SA$ and $E_\mu\colon \SA\to\SA$ is multiplication with $[\lambda]$, while $F_\mu\colon\SA\to\SA$ is the identity.  If $[\lambda]$ vanishes in $\SA$, then  $M_\mu=0$. 

Let us go one step further. Suppose that $[\lambda]$ did not vanish in $\SA$. We would like to find an object $N$ that extends the object $M$ just defined to the weight $\mu:=\lambda-4$. We now have $N_{\delta\mu}=\widehat N_\mu=\SA\oplus\SA$ and $\widehat F_\mu\colon N_{\delta\mu}\to \widehat N_\mu$ is again the identity. Now let $v$ be a generator of $N_\lambda=M_\lambda=\SA$. Then $v$ and $F_{1}v$ form a basis of $N_{\delta\mu}$, and $\widehat F_2 v$, $\widehat F_1 F_1 v$  a basis of $\widehat N_\mu$. Using the commutation relations we obtain
\begin{align*}
\widehat E_1\widehat F_2 v&=F_2 E_1 v + \qchoose{\lambda-4+1+2}{1}F_1 v=0+ [\lambda-1] F_1v,\\
\widehat E_1\widehat F_1 F_1 v&= F_1 E_1 F_1 v +  \qchoose{\lambda-4+1+1}{1} F_1 v=[\lambda]F_1 v+[\lambda-2] F_1 v\\
\widehat E_2\widehat F_2 v &= F_2 E_2 v+\qchoose{\lambda-4+2+2}{1} F_1E_1v+\qchoose{\lambda-4+2+2}{2} v\\
&=0+0+\qchoose{\lambda}{2}v,\\
\widehat E_2\widehat F_1 F_1 v &= F_1 E_2 F_1 v+\qchoose{\lambda-4+2+1}{1}E_1 F_1 v\\
&=0+[\lambda-1][\lambda]v.
\end{align*}
Hence $\widehat E_\mu$ is given by the matrix $\left(\begin{matrix}[\lambda-1] & \frac{[\lambda][\lambda-1]}{[2]} \\ [\lambda]+[\lambda-2]&[\lambda][\lambda-1]\end{matrix}\right)$. Suppose that $[2]$ is invertible in $\SA$. A short calculation shows $\frac{[\lambda]}{[2]}([\lambda]+[\lambda-2])=[\lambda][\lambda-1]$. So the second column is $\frac{[\lambda]}{[2]}$ times the first column, hence the kernel of $\widehat E_\mu$ has at least rank $1$.  It has rank $2$  if $[\lambda-1]$ vanishes in $\SA$. If the kernel has rank $1$, then $N_{\lambda-4}\cong \SA$ and $E_{\lambda-4}$ is given by the first column of the above matrix. If the rank of the kernel is $2$, then $N_{\lambda-4}$ vanishes. 
\subsection{Functoriality} 

In this section we show that the construction of the object $M$ in Proposition \ref{prop-minimalextobj} yields a functor $\mathsf{E}\colon \CX_{I^\prime}\to\CX_I$.
\begin{proposition}\label{prop-minextfun} Let $I^\prime\subset I$ be a pair of closed subsets. Then there exists a functor $\mathsf{E}=\mathsf{E}_{I^\prime}^I\colon \CX_{I^\prime}\to\CX_I$ with the following properties.
\begin{enumerate}
\item The functor $\mathsf{E}$ is left adjoint to the restriction functor $(\cdot)_{I^\prime}\colon\CX_I\to\CX_{I^\prime}$.
\item The canonical transformation $\id_{\CX_{I^\prime}}\to(\cdot)_{I^\prime}\circ \mathsf{E}$ coming from the adjointness in (1) is an isomorphism of functors.
\item The functorial homomorphism $\Hom_{\CX_{I^\prime}}(M^\prime,N^\prime)\to \Hom_{\CX_I}(\mathsf{E}(M^\prime),\mathsf{E}(N^\prime))$ is an isomorphism.
\item For all $\mu\in I\setminus I^\prime$ and all $M^\prime$ in $\CX_{I^\prime}$ we have $\mathsf{E}(M^\prime)_{\mu}=\im F_\mu$ and hence $\mathsf{E}(M^\prime)_{\{\mu\}}=\mathsf{E}(M^\prime)_{\lgl\mu\rgl}$ and $\mathsf{E}(M^\prime)_\mu\ne 0$ implies that there exists a weight  $\lambda$  of $M^\prime$ with $\mu\le\lambda$. 
\end{enumerate}
\end{proposition}

Due to property (4) we call $\mathsf{E}_{I^\prime}^I(M)$ the {\em minimal extension} of $M$. 
\begin{proof} For any object $M^\prime$ in $\CX_{I^\prime}$ we fix an object $M$ satisfying the properties in Proposition \ref{prop-minimalextobj}. We also fix an (arbitrary) isomorphism $\tau^{M^\prime}\colon M^\prime\xrightarrow{\sim} M_{I^\prime}$. Then we set $\mathsf{E}(M^\prime)=M$. Now let $f\colon M^\prime\to N^\prime$ be a homomorphism in $\CX_{I^\prime}$. Then Proposition \ref{prop-minimalextobj} shows that the homomorphism $\Hom_{\CX_I}(\mathsf{E}(M^\prime),\mathsf{E}(N^\prime))\to\Hom_{\CX_{I^\prime}}(\mathsf{E}(M^\prime)_{I^\prime},\mathsf{E}(N^\prime)_{I^\prime})$ that is given by the functor $(\cdot)_{I^\prime}$ is an isomorphism. The fixed identifications $\tau^{M^\prime}$ and $\tau^{N^\prime}$ allow us to obtain an isomorphism $\Hom_{\CX_I}(\mathsf{E}(M^\prime),\mathsf{E}(N^\prime))\xrightarrow{\sim}\Hom_{\CX_{I^\prime}}(M^\prime,N^\prime)$. We define the homomorphism $\mathsf{E}(f)\colon\mathsf{E}(M^\prime)\to\mathsf{E}(N^\prime)$ as the preimage of $f$. A moment's thought shows that the map $f\mapsto \mathsf{E}(f)$ is compatible with compositions and respects the identity morphisms. Hence we indeed obtain a functor $\mathsf{E}\colon \CX_{I^\prime}\to\CX_I$. 

We need to show that this functor has the required properties. By our construction, the fixed isomorphisms  $\tau^{M^\prime}\colon M^\prime\to M_{I^\prime}$ yield a natural transformation $\id_{\CX_{I^\prime}}\to (\cdot)_{I^\prime}\circ\mathsf{E}$, hence (2). This natural transformation induces a homomorphism 
$$
\Hom_{\CX_I}(\mathsf{E}(M^\prime),N)\to\Hom_{\CX_{I^\prime}}(M^\prime,N_{I^\prime})
$$
that is functorial in both $M^\prime$ and $N$. As $\tau^{M^\prime}$ is an isomorphism, it follows from Proposition \ref{prop-minimalextobj} that the above is a bijection, i.e.~$\mathsf{E}$ is left adjoint to $(\cdot)_{I^\prime}$, hence (1). Finally, (3) follows from the construction of the functor $\mathsf{E}$, and (4) is property (2) in Proposition \ref{prop-minimalextobj}.
\end{proof}

\subsection{The category of minimal objects}
The properties of the minimal extension functors allow us now to classify all minimal objects in $\CX$.
\begin{proposition}\label{prop-catminimal} Suppose that $\SA$ is a unital Noetherian $\SZ$-domain.
\begin{enumerate}
\item  For all $\lambda\in X$ there exists an up to isomorphism unique object $S_{\min}(\lambda)$ in $\CX$ with the following properties.
\begin{enumerate}
\item $S_{\min}(\lambda)_\lambda$ is free of rank $1$ and $S_{\min}(\lambda)_\mu\ne\{0\}$ implies $\mu\le\lambda$. 
\item $S_{\min}(\lambda)$ is indecomposable and minimal. 
\end{enumerate}
\end{enumerate}
Moreover, the objects $S_{\min}(\lambda)$ characterized in (1) have the following properties.
\begin{enumerate}\setcounter{enumi}{1}
\item For all $\lambda\in X$ we have $\End_{\CX}(S_{\min}(\lambda))=\SA\cdot\id$, and $\Hom_{\CX}(S_{\min}(\lambda),S_{\min}(\mu))=0$ for $\lambda\ne\mu$.
\item Let $S$ be a  minimal  object in $\CX$. Then there is an index set $J$ and some elements $\lambda_i\in X$ for $i\in J$ such that $S\cong \bigoplus_{i\in J}S_{\min}(\lambda_i)$. 
\end{enumerate}
\end{proposition}

Sometimes we will write $S_{\min,\SA}(\lambda)$ to incorporate the ground ring.

\begin{proof} We start with proving that there exists an object $S_{\min}(\lambda)$ satisfying the properties (1a) and (1b) as well as (2).  We then show that (3) holds with this particular set of objects $S_{\min}(\lambda)$, which then implies the remaining uniqueness statement in (1). So  set ${I_\lambda}:=\{\mu\in X\mid \lambda\le\mu\}$. This is a closed subset of $X$ that contains $\lambda$ as a minimal element. Define $S^\prime=\bigoplus_{\mu\in {I_\lambda}}S^\prime_\mu$ by setting $S^\prime_\lambda=\SA$ and $S^\prime_\mu=\{0\}$ for $\mu\in I_\lambda\setminus\{\lambda\}$.  For any $\mu\in I_\lambda$, $\alpha\in\Pi$, $n>0$ we have  $S^\prime_{\mu+n\alpha}=0$, and so all the maps $F_{\mu,\alpha,n}$ and $E_{\mu,\alpha,n}$ are the zero homomorphisms. Then $S^\prime$ is an object in $\CX_{I_\lambda}$. We set $S_{\min}(\lambda):=\mathsf{E}_{I_\lambda}^X(S^\prime)$. 

Proposition \ref{prop-minextfun}, (4) yields that  $S_{\min}(\lambda)_\mu=0$ unless $\mu\le\lambda$. As $S_{\min}(\lambda)_\lambda=S^\prime_\lambda=\SA$, $S_{\min}(\lambda)$ satisfies (1a). Proposition \ref{prop-minextfun} yields  $\End_{\CX}(S_{\min}(\lambda))=\End_{\CX_{I_\lambda}}(S^\prime)=\SA\cdot\id$, hence $S_{\min}(\lambda)$ is indecomposable, and property (2) holds. The same proposition also implies that $S_{\min}(\lambda)_{\{\mu\}}=S_{\min}(\lambda)_{\lgl\mu\rgl}$ for all $\mu\in X\setminus I_\lambda$. One checks directly that this identity is also satisfied for all $\mu\in I_\lambda$, hence $S_{\min}(\lambda)$ is minimal.

We are now left with proving  property (3), where we assume that the  $S_{\min}(\lambda)$ appearing in the statement are the objects we just constructed explicitely.  So let $S$ be a minimal  object, and fix a maximal weight $\lambda$ of $S$ (which exists by (X1)). Then $S_\lambda$ is a free $\SA$-module by assumption (X3c).  The maximality of $\lambda$ implies that the restriction $S_{I_\lambda}$ in $\CX_{I_\lambda}$ is isomorphic to a direct sum of copies of the object $S^\prime$ defined above.  Since  $S$ and  $\mathsf{E}_{I_\lambda}^X(\bigoplus S^\prime)$ are both minimal,   Lemma \ref{lemma-minimalextmor} implies  that the isomorphism $S_{I_\lambda}\cong \bigoplus S^\prime$  extends, so there are morphisms $f\colon  \mathsf{E}_{I_\lambda}^X (\bigoplus S^\prime)\to S$ and $g\colon  S\to \mathsf{E}_{I_\lambda}^X(\bigoplus S^\prime)$ such that $(g\circ f)|_{\lambda}$ is the identity. Now $\mathsf{E}_{I_\lambda}^X( \bigoplus S^\prime)\cong \bigoplus S_{\min}(\lambda)$. Using the already proven property (2) we deduce that  $g\circ f$ is an automorphism. Hence $ \bigoplus S_{\min}(\lambda)$ is isomorphic to a direct summand of $S$. By construction, $\lambda$ is not a weight of a direct complement. From here we can continue by induction to prove (3). 
\end{proof}

\subsection{Semisimplicity in the field case}
Now assume for a moment that $\SA=\SK$ is a field. 

\begin{lemma}\label{lemma-fieldcase} Any object in $\CX_\SK$ is minimal and maximal. Hence any object in $\CX_\SK$  is isomorphic to a direct sum of copies of  the objects $S(\lambda):=S_{\min}(\lambda)$ for various $\lambda$. In particular, $\CX_\SK$ is a semi-simple category.
\end{lemma}
\begin{proof}
As $\SK$ is a field, no torsion occurs, hence $M_{\{\mu\}}=M_{\lgl\mu\rgl}=M_{\{\mu\},\max}$, so $M$ is both minimal and maximal. The claims now follow all from Proposition \ref{prop-catminimal} and the fact that $\End_{\CX_\SK}(S_{\min}(\lambda))=\SK\cdot \id$. 
\end{proof}

\section{Representations of quantum groups}\label{sec-qg}

In this section we show that the category $\CX$ has an interpretation in terms of quantum group representations. The main reasons for this are the uniqueness of the indecomposable minimal objects $S_{\min}(\lambda)$ and the semisimplicity statement in Lemma \ref{lemma-fieldcase} in the field case. 

\subsection{Quantum groups over $\SZ$-algebras}
We denote by $U_\SZ$ the quantum group over $\SZ=\DZ[v,v^{-1}]$ (with divided powers) associated with the Cartan matrix $(\lgl\alpha,\beta^\vee\rgl)_{\alpha,\beta\in\Pi}$ of $R$. Its definition by generators and relations can be found  in \cite[Sections 1.1-1.3]{L90}. We denote by $e^{[n]}_\alpha,  f^{[n]}_\alpha, k_\alpha, k_\alpha^{-1}$ for $\alpha\in\Pi$ and $n>0$ the standard generators of $U_\SZ$. 

For  $\alpha\in R$,  $n>0$ also the element
$$ 
\qchoose{k_\alpha}{n}_{\alpha}:=\prod_{s=1}^n \frac{k_\alpha v_\alpha^{-s+1}-k_{\alpha}^{-1}v_\alpha^{s-1}}{v_\alpha^{s}-v_\alpha^{-s}} 
$$
is contained in $U_\SZ$ (where $v_\alpha:=v^{d_\alpha}$). We let $U_\SZ^+$, $U_\SZ^-$ and $U_\SZ^0$ be the unital subalgebras of $U_\SZ$ that are generated by the sets $\{e_\alpha^{[n]}\}$, $\{f_\alpha^{[n]}\}$ and $\{k_\alpha, k_\alpha^{-1},\qchoose{k_\alpha}{n}_{\alpha}\}$, resp. A remarkable fact, proven by Lusztig, is that each of these subalgebras is free over $\SZ$ and admits a PBW-type basis, and that the multiplication map $U^-_\SZ\otimes_\SZ U_\SZ^0\otimes_\SZ U_\SZ^+\to U_\SZ$ is an isomorphism of $\SZ$-modules (Theorem 6.7 in \cite{L90}). 

For a unital $\SZ$-algebra $\SA$  we set $U_\SA:=U_\SZ\otimes_\SZ\SA$ and $U_\SA^\ast:= U_\SZ^\ast\otimes_\SZ\SA$ for $\ast=-,0,+$.
In this article we consider $U_\SA$ only as an associative, unital algebra and forget about the Hopf algebra structure. For better readability we write $U$, $U^+$,\dots instead of $U_\SA$, $U_\SA^+$,\dots, once the $\SZ$-algebra $\SA$ is fixed.

\subsection{The category $\CO$} 
By \cite[Lemma 1.1]{APW} every $\mu\in X$ yields a character
\begin{align*}
\chi_\mu\colon U^0_\SZ&\to\SZ\\
k_\alpha^{\pm1}&\mapsto v_\alpha^{\pm \lgl\mu,\alpha^\vee\rgl}\\
\qchoose{k_\alpha}{r}_\alpha&\mapsto \qchoose{\lgl\mu,\alpha^\vee\rgl}{r}_{d_\alpha}\text{ ($\alpha\in\Pi$, $r\ge0$)}.
\end{align*}
We can extend the above character to a character $\chi_\mu\colon U^0_\SA\to\SA$. 
A $U_\SA$-module $M$ is called a {\em weight module} if $M=\bigoplus_{\mu\in X}M_\mu$, where
$$
M_\mu:=\{m\in M\mid H.m=\chi_\mu(H)m\text{ for all $H\in U_\SA^0$}\}.
$$
Hence all the weight modules that we consider in this article  are ``of type 1'' (cf.~\cite[Section 5.1]{JanQG}). An element $\mu\in X$ is called a {\em weight of $M$} if $M_\mu\ne\{0\}$. Now suppose that $\SA$ is also Noetherian. 

\begin{definition}\label{def-Q} Let $\CO=\CO_\SA$ be the full subcategory of the category of $U_\SA$-modules that contains all objects $M$ with the following properties.
\begin{enumerate}
\item $M$ is a weight module and its set of weights is quasi-bounded. 
\item For each $\mu\in X$, the weight space $M_\mu$ is a finitely generated torsion free $\SA$-module. 
\end{enumerate}
\end{definition}

\begin{remark} Note that the definition above yields an $\SA$-linear category, which in general is not abelian (due to the torsion freeness assumption). It is closed under taking subobjects (due to the fact that we assume that $\SA$ is Noetherian).   If $\SA=\SK$ is a field, then torsion freeness is always satisfied, and we obtain an abelian category.
\end{remark}

Now we establish a first, rather easy, link to the objects that we considered in the earlier chapters.
Let us denote by $\CX^{\pre}=\CX^{\pre}_\SA$ the category whose objects are  $X$-graded $\SA$-modules $M$ endowed with operators $E_{\alpha,n}$ and $F_{\alpha,n}$ as in Section \ref{sec-Xgradop} that satisfy conditions (X1) and (X2) (but not necessarily (X3)), and with morphisms being the $X$-graded $\SA$-linear homomorphisms that commute with the $E$- and $F$-maps.  

Let $M$ be an object in $\CO$. Let us denote  by $E_{\mu,\alpha,n}\colon M_\mu\to M_{\mu+n\alpha}$ and $F_{\mu,\alpha,n}\colon M_{\mu+n\alpha}\to M_\mu$ the homomorphisms given by the actions of $e_\alpha^{[n]}$ and $f_{\alpha}^{[n]}$, resp. By forgetting structure, we now consider $M$ only as an $X$-graded space endowed with these operators. 
\begin{lemma} \label{lemma-funS}  The above yields a  fully faithful functor
$$
\mathsf{S}\colon \CO\to\CX^{\pre}.
$$
\end{lemma}
\begin{proof} It is clear that the above construction is functorial.  Let $M$ be an object in $\CO$. We need to check that the graded space with operators that we obtain from $M$ satisfies the conditions (X1) and (X2). Condition (X1) is part of the definition of $\CO$. Now we check condition (X2). The case $\alpha\ne\beta$ follows from the fact that $e_\alpha^{[m]}$ and $f_{\beta}^{[n]}$ commute. Now we treat the case $\alpha=\beta$. Set
$$
\qchoose{k_\alpha;c}{r}_\alpha =\prod_{s=1}^r\frac{k_\alpha v_\alpha^{c-s+1}-k_{\alpha}^{-1}v_\alpha^{-c+s-1}}{v_\alpha^{s}-v_\alpha^{-s}}.
$$
This element is contained in $U^0$ and acts as multiplication with 
\begin{align*}
\prod_{s=1}^r\frac{v_\alpha^{\lgl\nu,\alpha^\vee\rgl+c-s+1}-v_\alpha^{-\lgl\nu,\alpha^\vee\rgl-c+s-1}}{v_\alpha^{s}-v_\alpha^{-s}}.
\end{align*}
on each vector of weight $\nu$. By \cite[Section 6.5]{L90} the following relations holds in $U_\SZ$ for all  $m,n>0$: 
 $$
 e_{\alpha}^{[m]}f_\alpha^{[n]}=\sum_{r=0}^{\min(m,n)} f_\alpha^{[n-r]}\qchoose{k_\alpha;2r-m-n}{r}_\alpha e_\alpha^{[m-r]}.
 $$
For  $v\in M_\mu$ we hence obtain
\begin{align*}
 e_{\alpha}^{[m]}f_\alpha^{[n]}(v)&=\sum_{r=0}^{\min(m,n)} f_\alpha^{[n-r]}\prod_{s=1}^r\frac{v_\alpha^{\zeta-s+1}-v_\alpha^{-\zeta+s-1}}{v_\alpha^{s}-v_\alpha^{-s}}e_\alpha^{[m-r]}(v),
 \end{align*}
 where $\zeta=\lgl\mu+(m-r)\alpha,\alpha^\vee\rgl+2r-m-n=\lgl\mu,\alpha^\vee\rgl+m-n$. In order to prove that condition (X2) holds, it remains to show that 
 $$
 \qchoose{\zeta}{r}_{d_\alpha}=\prod_{s=1}^r\frac{v_\alpha^{\zeta-s+1}-v_\alpha^{-\zeta+s-1}}{v_\alpha^{s}-v_\alpha^{-s}},
 $$
 which is (almost) immediate from the definition. Hence $\mathsf{S}$ is indeed a functor from $\CO$ to $\CX^{\pre}$.
 
Now $U$ is generated by the elements $e_\alpha^{[n]}$, $f_\alpha^{[n]}$ for $\alpha\in\Pi$ and $n>0$ as an algebra over $U^0$ by the PBW-theorem.  As the actions of the $e^{[n]}$- and $f^{[n]}$-elements are encoded by the $E_n$- and $F_n$-homomorphisms, and as the action of $U^{0}$ is  encoded by the $X$-grading, the functor $\mathsf{S}$ is fully faithful. 
\end{proof}
Note that we can consider $\CX$ as a full subcategory of $\CX^{\pre}$. In the next section we construct a functor $\mathsf{R}\colon\CX\to\CO$ that is right inverse to $\mathsf{S}$.

\subsection{A functor from $\CX$ to  $\CO$} First we suppose that $\SA=\SK$ is a field. 
For each $\lambda$, the character $\chi_\lambda$ of $U^0_\SK$ can be uniquely extended to a character of  $U_\SK^0U_\SK^+$ such that $\chi_\lambda(e_\alpha^{[n]})=0$ for all $\alpha\in\Pi$, $n>0$. So we obtain an $U_\SK^0U_\SK^+$-module $\SK_\lambda$ of dimension $1$. We denote by $\Delta_\SK(\lambda):=U_\SK\otimes_{U_\SK^0U_\SK^+}\SK_\lambda$ the induced $U_\SK$-module (this is the {\em Verma module} with highest weight $\lambda$). It  has a unique irreducible quotient that  we denote by   $L_\SK(\lambda)$. Both are objects in $\CO_\SK$. For more information on these objects, see \cite{A}, or, in the case $q=1$, \cite{FracPerO}. Recall  that in the field case a minimal object is also maximal and we  set $S_\SK(\lambda):=S_{\min,\SK}(\lambda)$ (cf.~Lemma \ref{lemma-fieldcase}). 

\begin{proposition} \label{prop-irreps} Suppose that $\SA=\SK$ is a field. 
\begin{enumerate}
\item For all $\lambda\in X$ the object  $\mathsf{S}(L_\SK(\lambda))$ is contained in the subcategory $\CX_\SK$ of $\CX^{\pre}_\SK$ and it is isomorphic  to $S_\SK(\lambda)$. 
\item The functor $\mathsf{S}$ induces an equivalence between the full subcategory $\CO^{ss}_\SK$ of semi-simple objects in $\CO_\SK$ and the category $\CX_\SK$.
\end{enumerate}
\end{proposition}

\begin{proof} We prove part (1). In view of Lemma \ref{lemma-funS} we need to check property (X3) in order to prove the first statement. Set $M=\mathsf{S}(L_\SK(\lambda))$. As $L_\SK(\lambda)$ is an irreducible $U_\SK$-module of highest weight $\lambda$,  it is cyclic as a $U^-$-module, hence we have $\im F_\mu=M_\mu$ for all $\mu\ne \lambda$, and $\im F_\lambda=\{0\}$. On the other hand, there are no non-trivial primitive vectors in $L_\SK(\lambda)$ of weight $\mu$ if $\mu\ne\lambda$. (A primitive vector is a vector annihilated by all $e_{\alpha}^{[n]}$.) So we have  $\ker E_\mu=\{0\}$ for $\mu\ne\lambda$ and $\ker E_\lambda=M_\lambda$. In any case we have $M_\mu=\ker E_\mu\oplus \im F_\mu$, which, by Lemma \ref{lemma-X3}, is equivalent to the set of conditions (X3). Hence $M$ is an object in $\CX_\SK$. Lemma \ref{lemma-fieldcase} now yields that $M$ is isomorphic to a direct sum of various $S_{\SK}(\mu)$'s. As $\mathsf{S}$ is fully faithful, $M$ is indecomposable, and a comparison of weights shows $M\cong S_{\SK}(\lambda)$. 

Now (2) follows from the fact that $\CX_\SK$ is a semi-simple category with simple objects $S_\SK(\lambda)$ by Lemma \ref{lemma-fieldcase}  and the fact that $\mathsf{S}$ is fully faithful.
\end{proof}

Here is our ``realization theorem''. 
\begin{theorem}\label{thm-XU} Let $\SA$ be a unital Noetherian domain that is a $\SZ$-algebra. Let $M$ be an object in $\CX=\CX_\SA$. Then there exists a unique $U$-module structure  on $M$ such that the following holds.
\begin{itemize}
\item The $X$-grading $M=\bigoplus_{\mu\in X}M_\mu$ is the weight decomposition.
\item  For all $\mu\in X$, $\alpha\in\Pi$, $n>0$, the homomorphisms $E_{\mu,\alpha,n}$ and $F_{\mu,\alpha,n}$ are the action maps of $e_{\alpha}^{[n]}$ on $M_\mu$ and $f_{\alpha}^{[n]}$ on $M_{\mu+n\alpha}$, resp.
\end{itemize}
 From this we obtain a fully faithful functor $\mathsf{R}\colon\CX\to \CO$, and we have $\mathsf{S}\circ\mathsf{R}\cong\id_{\CX}$. 
\end{theorem}

\begin{proof} Note that the uniqueness statement in the claim above follows immediately from the fact that $U$ is generated as an algebra by the elements $e_\alpha^{[n]}$, $f_\alpha^{[n]}$ and $k_\alpha$, $k_\alpha^{-1}$ for $\alpha\in\Pi$ and $n>0$. 
We now prove the existence of a $U$-module structure on $M$ with the alleged properties. 

First suppose that $\SA=\SK$ is a field. Then every object in $\CX_\SK$ is isomorphic to a direct sum of various $S_{\SK}(\lambda)$'s by Lemma \ref{lemma-fieldcase}. By  Proposition \ref{prop-irreps} we have $\mathsf S(L_\SK(\lambda))\cong S_{\SK}(\lambda)$. Hence  any $S_{\SK}(\lambda)$ carries the structure of an  $U_\SK$-module of the required kind  (making it isomorphic to $L_\SK(\lambda)$). So the result holds in the case that $\SA$ is a field.

Now let $\SA$ be a unital Noetherian domain. We denote by $\SK$ its quotient field. For any object $M$ in $\CX=\CX_\SA$, $M_\SK=M\otimes_\SA\SK$ is an object in $\CX_\SK$ by Lemma \ref{lemma-basechange}. As $M$ is a torsion free $\SA$-module by Lemma \ref{lemma-X3}, we can view $M$ as an $\SA$-submodule in $M_\SK$. Now by the above, we can view $M_\SK$ as an object in $\CO_\SK$. As $M$ is stable under the maps $E_{\alpha,n}$ and $F_{\alpha,n}$, it is stable under the action of $e_{\alpha}^{[n]}$ and $f_\alpha^{[n]}$. Moreover, it is clearly stable under the action of $k_\alpha$ and $k_\alpha^{-1}$. Hence it is stable under the action of  $U_\SA\subset U_\SK$. So there is indeed a natural $U_\SA$-module structure on $M$, and one immediately checks that this makes it into an object of category $\CO_\SA$.

Clearly the above $U_\SA$-structure depends functorially on $M$, so  we indeed obtain a functor $\mathsf{R}$ from $\CX_\SA$ to $\CO_\SA$. It is clearly fully faithful and obviously  $\mathsf{S}\circ\mathsf{R}$ is isomorphic to the identity on $\CX_\SA$. 
\end{proof}

\section{Objects admitting a Weyl filtration}\label{sec-Weyl}
 The main goal of this section is to show that the functors $\mathsf{S}$ and $\mathsf{R}$ induce mutually inverse equivalences between the category $\CX^{\fin}$ of objects in $\CX$ that are free of finite rank over $\SA$, and the category $\CO^W$ of objects in $\CO$ that admit a (finite) Weyl filtration. 
 
We need to assume that the quotient field of our ground ring $\SA$ is {\em generic}, i.e.~of  {\em quantum characteristic $0$} (see Definition \ref{def-generic}). The finite dimensional representation theory of $U_\SK$ is semi-simple in this case. For the most part of this section we do not need $\SA$ to be local. 

\subsection{Generic algebras}\label{subsec-groundrings}
Let $\SA$ be a unital $\SZ$-algebra that is a Noetherian domain.  

\begin{definition} \label{def-generic} We say that $\SA$ is {\em generic} (or of {\em quantum characteristic $0$}) if for all  $n\ne 0$ and all $d>0$ the image of the quantum integer $[n]_{d}$ in $\SA$ is non-zero (i.e.~invertible in the quotient field $\SK$).
\end{definition}
Recall that we denote by $q\in\SA$ the image of $v$ under the structural homomorphism $\SZ\to\SA$, $f\mapsto f\cdot 1_\SA$. Note that if $\SA$ is not generic, then $q$ is a root of unity in $\SA$,  as 
\begin{align*}
[n]_d&=\frac{v^{dn}-v^{-dn}}{v^d-v^{-d}}
=v^{-dn+d}\frac{v^{2dn}-1}{v^{2d}-1}. 
\end{align*}
The converse is not true. For example, a field of characteristic $0$ with $q=1$ is generic. However,  a field of positive characteristic and $q=1$ is not generic, but the ring of $p$-adic integers $\DZ_p$ with $q=1$ is generic. If $\zeta\in \DC$ is a root of unity of order $>2$, then $\SA=\DQ(\zeta)$ with $q=\zeta$ is not generic. 

Let $p$ be a prime number and denote by $\SZ_\fp$ the localization of $\SZ$ at the prime ideal 
\begin{align*}
\fp&:=\{g\in\SZ\mid \text{ $g(1)$ is divisible by $p$}\}\\
&=\ker\left( \SZ\xrightarrow{v\mapsto 1} \DF_p\right),
\end{align*}
i.e.~$\SZ_\fp=\{\frac{f}{g}\in\Quot\SZ\mid \text{ $g(1)$ is not divisible by $p$}\}$. So $\SZ_\fp$ with $q=v$ is a  local and generic $\SZ$-algebra, and its residue field is $\DF_p$ (with $q=1$). Similarly, denote by $\sigma_l\in\DQ[v]$ the $l$-th cyclotomic polynomial, and let $\DQ[v]_{(\sigma_l)}$ be the localization at the prime ideal generated by $\sigma_l$. We obtain a  local and generic $\SZ$-algebra with $q=v$, with residue field $\DQ[\zeta_l]$, the $l$-th cyclotomic field (with $q=\zeta_l$).

The assumption that $\SA$ is generic is not a decisive restriction even if one is interested in the  theory of tilting modules for  algebraic groups in positive characteristics or for  quantum groups at roots of unity. As we will show in this article, these objects admit deformations to the generic and local algebras $\SZ_\fp$ and $\DQ[v]_{(\sigma_l)}$ and  one obtains information in the non-generic cases by specializing results from generic cases.

\subsection{Modules admitting a Weyl filtration}
 Let us now assume that $\SA$   is  generic. We denote by  $\SK$ its quotient field. Recall that for any $\lambda\in X$ we denote by $L_\SK(\lambda)$ the irreducible object in $\CO_\SK$  with highest weight $\lambda$. Then $L_\SK(\lambda)$ is finite dimensional if and only if $\lambda$ is dominant, i.e.~is contained in the set $X^+=\{\lambda\in X\mid \lgl\lambda,\alpha^\vee\rgl\ge 0\text{ for all $\alpha\in\Pi$}\}$ of dominant weights, cf.~\cite[Theorem 2.3]{A}. If $\lambda$ is dominant, then we define $W(\lambda)=W_\SA(\lambda)$ as the $X$-graded $U_\SA$-submodule in $L_\SK(\lambda)$ generated by a non-zero element in $L_\SK(\lambda)_\lambda$. This is an object in $\CO_\SA$ and it does not depend, up to  isomorphism, on the choice of the element.  It is called the {\em Weyl module} with highest weight $\lambda$. 
\begin{proposition}\label{prop-Weylchar}  Let $\SA$ be local and generic, and let $\lambda\in X$ be dominant. Then $W(\lambda)$ is a free $\SA$-module of finite rank and its character is given by Weyl's character formula.
\end{proposition}
\begin{proof} The above statement is proven for a special choice of $\SA$ in \cite[Proposition 1.22]{APW}. We claim that their proof carries over to the more general case above. As $\SA$ is generic, the character of $L_\SK(\lambda)=W_\SA(\lambda)\otimes_\SA\SK$ is given by Weyl's character formula. We denote by $\SF$ the residue field of $\SA$. As in loc.cit. it is now sufficient to show that for all weights $\mu$ the $\SF$-dimension  of  $W_\SA(\lambda)_\mu\otimes_\SA\SF$ is at most the $\SK$-dimension of $W_\SA(\lambda)_\mu\otimes_\SA\SK$. But the $U_\SF$-module $W_\SA(\lambda)\otimes_\SA\SF$ is finite dimensional and of highest weight $\lambda$. Hence (Remark 4.2 in \cite{RH}) it is a quotient of the $U_\SF$-module $D_\SF(\lambda)$ that is obtained from the $U_\SF^+$-module $\SF_\lambda$ by Joseph's induction functor. But the character of $D_\SF(\lambda)$ is given by Weyl's character formula as well (see \cite[Theorem 5.4, Corollary 5.8]{RH}), hence the character of $W_\SA(\lambda)\otimes_\SA\SF$ is bounded from above by Weyl's character formula. 
\end{proof}

\begin{definition} We say that $M$ {\em admits a Weyl filtration} if there is a finite filtration $0=M_0\subset M_1\subset \dots\subset M_n=M$ and $\lambda_1,\dots,\lambda_
n\in X^+$ such that for each $i=1,\dots,n$, the subquotient $M_i/M_{i-1}$ is isomorphic to $W(\lambda_i)$. 

\end{definition}
We denote by $\CO^W=\CO_\SA^W$ the full subcategory of $\CO$ that contains all objects that admit a Weyl filtration. Note that if $\SA=\SK$ is a generic  field, then  $\CO^W_\SK$ is a semi-simple category (cf.~Theorem 5.15 and Section 6.26 in \cite{JanQG}).

\subsection{A criterium for Weyl filtrations} Suppose that $\SA$ is  local and generic. 
 Let $M$ be a module for $U=U_\SA$.

\begin{lemma}\label{lemma-freehws}  Suppose that  $M$ is finitely generated as an $\SA$-module and  that there exists a dominant element $\lambda\in X$ such that the following holds.
\begin{enumerate}
\item The weight space $M_\lambda$ generates  $M$ as a $U^-$-module.
\item The weight space $M_\lambda$ is a free $\SA$-module of finite rank $r$. 
\end{enumerate}
Then $M$ is isomorphic to a direct sum of $r$ copies of $W(\lambda)$. 
\end{lemma}

\begin{proof} First let us show that $M$ is free over $\SA$. Again we denote by $\SK$ and $\SF$ the quotient field and the residue field of $\SA$, resp. Consider the $U_\SK$-module $M_\SK$. It is of finite $\SK$-dimension, and generated by its $\lambda$-weight space, which is of $\SK$-dimension $r$. As $\SA$ is generic, $M_\SK$ is isomorphic to a direct sum of $r$ copies of $L_\SK(\lambda)$. Now consider the $U_\SF$-module $M_\SF$. Again, this is of finite $\SF$-dimension and generated by its $\lambda$-weight space, which is of $\SF$-dimension $r$. As in the proof of Proposition \ref{prop-Weylchar} we deduce  that $M_\SF$ is a quotient of a direct sum of $r$ copies of $D_\SF(\lambda)$. As the characters of $L_\SK(\lambda)$ and $D_\SF(\lambda)$ are both given by Weyl's character formula, we deduce, again as in the proof of Proposition \ref{prop-Weylchar}, that each weight space of $M$ is free over $\SA$ (and that its character is given by $r$ times Weyl's character formula). 

Hence $M$ embeds into $M_\SK\cong L_\SK(\lambda)^{\bigoplus r}$. As $M$ is generated by $M_\lambda$, which is an $\SA$-lattice in $(M_\SK)_\lambda$ of rank $r$, $M$ is isomorphic to a direct sum of $r$ copies of the Weyl module $W_\SA(\lambda)$. 
\end{proof}

Let $M$ be an object in $\CO$. 
For $\lambda\in X$ define $M[\lambda]\subset M$ as the $U$-submodule of $M$ that is generated by all weight spaces $M_\mu$ such that $\mu\not\le\lambda$. Then  $M/M[\lambda]$ is the largest quotient of $M$ that has the property that all of its weights are smaller or equal to $\lambda$. Set 
$M_{[\lambda]}=(M/M[\lambda])_\lambda$. Clearly, if $N\subset M$ is a submodule, then $N[\lambda]$ is a submodule of $M[\lambda]$. Also, $M[\lambda]_\mu=M_\mu$ for all $\mu\not\le\lambda$. 
\begin{lemma}\label{lemma-ws} For $\nu\not\le\lambda$ the inclusion $M[\lambda][\nu]\subset M[\nu]$ is an isomorphism on the $\nu$-weight space.
\end{lemma}
\begin{proof} For an object $B$ in $\CO$ let us denote by $B[{}^\prime\nu]$ the $U$-submodule of $B$ that is generated by the weight spaces $B_\mu$ with $\mu>\nu$. Then $B[{}^\prime\nu]\subset B[\nu]$. The PBW-theorem implies that $B[{}^\prime\nu]$ is generated by $\bigoplus_{\mu>\nu} B_\mu$ even over $U^-$. Likewise,  $B[\nu]$ is generated by $\bigoplus_{\mu\not\le\nu} B_\mu$ over $U^-$. Hence $B[{}^\prime\nu]_\nu=B[\nu]_\nu$. Now $M[\lambda][{}^\prime\nu]=M[{}^\prime\nu]$, as for any $\mu>\nu$ we have $\mu\not\le\lambda$, so $M[\lambda]_\mu=M_\mu$. Hence $M[\lambda][\nu]_\nu=M[\lambda][{}^\prime\nu]_\nu=M[{}^\prime\nu]_\nu=M[\nu]_\nu$.
\end{proof}
Now we state and prove a criterion for the existence of Weyl filtrations. 
\begin{proposition}\label{prop-charWf} Suppose that $\SA$ is generic.  Let $M$ be an object in $\CO$. The following statements are equivalent.
\begin{enumerate}
\item The set of weights of $M$ is finite and  $M_{[\nu]}$ is a free $\SA$-module of finite rank for all $\nu\in X$. 
\item  $M$ admits a Weyl filtration.
\end{enumerate}
If either of the above holds, then the multiplicity of $W(\mu)$ in a Weyl filtration equals the rank of $M_{[\mu]}$. In particular, $M_{[\mu]}\ne 0$ implies that $\mu$ is dominant. 
\end{proposition}

\begin{proof}  Assume that (2) holds. Then  the set of weights of $M$  is finite. Standard arguments show that  $\Ext^1_{\CO}(W(\lambda),W(\mu))=0$ if $\mu\not>\lambda$ (cf. the proof of  Lemma \ref{lemma-freehws}). Let $\nu\in X$. The $\Ext$-vanishing statement now implies that the subquotients of a given Weyl filtration of $M$ can be ``rearranged'' in such a way that  we obtain a filtration   $0=M_0\subset M_1\subset\dots\subset M_n=M$ and some $1\le r\le s\le n$ such that $M_i/M_{i-1}$ has a highest weight $\not\le\nu$ if $i\le r$, has highest weight $\nu$ if $r< i\le s$, and has a highest weight $<\nu$, if $i>s$. Hence $M_s/M_r$ is a direct sum of copies of $W(\nu)$. 
Moreover, in the notation of the paragraph preceding this proposition, we have $M[\nu]=M_r$ and $M_{[\nu]}=(M/M_r)_\nu=(M_s/M_r)_{\nu}$ (as $(M/M_s)_\nu=\{0\}$). Hence $M_{[\nu]}$ is free of finite rank $s-r$ as an $\SA$-module. So (2) implies (1). 

Now assume that (1) holds. We can assume that $M\ne 0$. As the set of weights of $M$ is finite there  must be a minimal weight $\lambda$ such that $M_{[\lambda]}\ne 0$ (note that $M_{[\gamma]}=M_\gamma$ if $\gamma$ is maximal among the weights of $M$, so some $M_{[\gamma]}$ are non-zero). Let us fix such a minimal $\lambda$. Set  $N:=M[\lambda]\subset M$ and  $M^\prime=M/N$, so by definition $M^\prime_\lambda=M_{[\lambda]}$. We claim the following.
\begin{enumerate}
\item[(a)] We have $N_{[\nu]}\cong M_{[\nu]}$ for all $\nu\ne\lambda$, and $N_{[\lambda]}=0$. 
\item[(b)] $M^\prime$ is generated, as a $U^-_\SA$-module, by its $\lambda$-weight space. 
\item[(c)] $M^\prime$ is finitely generated as an $\SA$-module, and $\lambda$ is dominant. 
\end{enumerate}
If these statements are true, then we can prove that $M$ admits a Weyl filtration  as follows. From (a) we can deduce, by induction on the number of weights $\nu$ with $M_{[\nu]}\ne 0$, that $N$ admits a Weyl filtration.   Then from (b) and (c) we deduce, using the fact that $M^\prime_\lambda=M_{[\lambda]}$ is free of finite rank (by assumption) and Lemma \ref{lemma-freehws}, that $M^\prime$ is isomorphic to a direct sum of $\rk\, M_{[\lambda]}$ many copies of $W(\lambda)$. So both $N$ and $M^\prime=M/N$ admit a Weyl filtration, hence so does $M$. The last statement in the proposition about the multiplicities follows  by induction as well. 

So let us prove (a). By definition, $N$ is generated by  its weight spaces $N_\gamma$ with $\gamma\not\le\lambda$. For all $\nu\le\lambda$ we hence have $N[\nu]=N$, so $N_{[\nu]}=0$. The minimality of $\lambda$ implies $M_{[\nu]}=0$ for all $\nu<\lambda$. Now suppose that $\nu\not\le\lambda$. Then $N_\nu=M_\nu$ and $N[\nu]_\nu=M[\nu]_\nu$ by Lemma \ref{lemma-ws}, hence $N_{[\nu]}=(N/N[\nu])_\nu=(M/M[\nu])_\nu=M_{[\nu]}$. Part (a) is proven.

Now we prove (b). Recall that all weights of $M^\prime$ are smaller or equal to  $\lambda$.  If $M^\prime$ wasn't generated as an $U_\SA^-$-module by its $\lambda$-weight space, then there would exist a weight $\nu<\lambda$ such that $M^\prime_{[\nu]}\ne 0$. Hence $M^\prime$ is not generated  (over $U_\SA$) by $\bigoplus_{\gamma\not\le\nu}M^\prime_\gamma$. As $M^\prime$ is quotient of $M$, $M$ cannot be generated by $\bigoplus_{\gamma\not\le\nu}M_\gamma$. Hence  $M_{[\nu]}\ne 0$, which is a contradiction to the minimality of $\lambda$. 

 We turn to statement (c). From assumption (1) it follows that $M$ is finitely generated as an $\SA$-module. Hence so is its quotient $M^\prime$. By (b), $M^\prime$ is a quotient of a finite direct sum of copies of the Verma module $\Delta(\lambda)$ (cf.~the proof of Lemma \ref{lemma-freehws}). By Proposition 3.2 in \cite{L98}, a simple highest weight module for $U_\SK$ is integrable only if $\lambda$ is dominant. As $M^\prime\otimes_\SA\SK$ is finite dimensional,  $\lambda$ must be dominant.  
\end{proof}

\subsection{Finitely generated objects in $\CX$}
Now we want to show that the objects in $\CX$ that are finitely generated as $\SA$-modules, correspond, via the realization functor $\mathsf{R}$, to the objects in $\CO$ that admit a Weyl filtration. In a first step we are interested in what happens if we apply the functor $\mathsf{S}$ to objects that admit a Weyl filtration. 

\begin{proposition}\label{prop-equivcat} Suppose that $\SA$ is generic. 
\begin{enumerate}
\item Let $M$ be an object in $\CO^W$. Then $\mathsf{S}(M)$ is an object in $\CX$. 
\item For all dominant $\lambda$ we have $\mathsf{S}(W(\lambda))\cong S_{{\min}}(\lambda)$.
\end{enumerate}
\end{proposition}
\begin{proof} We prove claim (1).  In view of Lemma \ref{lemma-funS} we need to  show that the property (X3) is satisfied for $\mathsf{S}(M)$. So let $\mu\in X$. As $M_\SK$ is semi-simple we have  $(\mathsf{S}(M)_\mu)_\SK=(\im F_\mu)_\SK\oplus(\ker E_\mu)_\SK$. As $M$ admits a Weyl filtration, Proposition \ref{prop-Weylchar} implies that $M$ is free as an $\SA$-module. By Lemma \ref{lemma-X3}, we now only need to   show that property (X3c) holds. Let
$
\{0\}=M_0\subset M_1\subset \dots\subset M_n=M
$ be a filtration 
such that $M_{i+1}/M_i$ is isomorphic to $W(\mu_i)$. As in the proof of Proposition \ref{prop-charWf}  we can assume that   there exists an integer $r$ such that $\mu<\mu_i$ implies $i\le r$. It follows that  $\im F_\mu =(M_r)_\mu\subset M_\mu$. As the quotient $M/M_r$ admits a filtration with subquotients isomorphic to Weyl modules with dominant highest weights,  it is free as an $\SA$-module by Proposition \ref{prop-Weylchar}. In particular, its $\mu$-weight space is free. By the above, this identifies with $M_\mu/\im F_\mu$. Hence property (X3c) holds, so $\mathsf{S}(M)$ is an object in $\CX$. 

Now we prove claim (2). 
For $N=\mathsf{S}(W(\lambda))$ we have $N_\mu=\im F_\mu$ for all $\mu\ne \lambda$ as $W(\lambda)$ is cyclic as a $U^-$-module and $U^-$ is generated by the elements $f_{\alpha}^{[n]}$ with $\alpha$ simple and $n>0$. Hence  $N$ is an indecomposable minimal object, so $N\cong S_{{\min}}(\nu)$ for some $\nu\in X$ by Proposition \ref{prop-catminimal}.  A comparison of weights shows $\nu=\lambda$.
\end{proof} 
Now we define the counterpart of $\CO^W\subset\CO$ in $\CX$.
\begin{definition}\label{def-Xfin}
We denote by $\CX^{\fin}$ the full subcategory of $\CX$ that contains all objects $M$ that are finitely generated as $\SA$-modules. 
\end{definition}
We will see in a moment that any object in $\CX^{\fin}$ is automatically free as an $\SA$-module.
As we assume that each weight space of an object in $\CX$ is finitely generated, the property in the definition above is equivalent to the set of weights being finite.

\begin{theorem} \label{thm-Wflag} Suppose that $\SA$ is generic. Then the functors $\mathsf{S}$ and $\mathsf{R}$ restrict to mutually inverse equivalences between the categories $\CX^{\fin}$ and $\CO^{W}$.
\end{theorem}

\begin{proof}  In view of Proposition \ref{prop-equivcat} we need to show that $\mathsf{S}(M)$ is finitely generated as an  $\SA$-module for all objects $M$ in $\CO$ that admit a Weyl filtration and, conversely, that $\mathsf{R}(M)$  admits a Weyl filtration for all objects $M$ in $\CX^{\fin}$. The first statement follows easily from the facts that the functor $\mathsf{S}$ is the identity functor on the underlying $\SA$-modules and that each Weyl module is free of finite rank as an $\SA$-module. 

Now suppose that $M$ is an object in $\CX$ that is finitely generated as an $\SA$-module. We already know from Theorem \ref{thm-XU} that $\mathsf{R}(M)$ is an object in $\CO$. We want to employ Proposition \ref{prop-charWf}, so we need to check that $\mathsf{R}(M)$ has only finitely many  weights and that $\mathsf{R}(M)_{[\mu]}$ is a free $\SA$-module of finite rank for all $\mu\in X$. The first statement is clear. For the second, note that we can canonically identify  $\mathsf{R}(M)_{[\mu]}$ with $M_\mu/\im F_\mu$. The latter is, by definition of the category $\CX$, a free $\SA$-module, and of finite rank as $M$ is of finite rank.
\end{proof}

From the above we can deduce that each object in $\CX^{\fin}$ is even {\em free} of finite rank as an $\SA$-module.

\section{The maximal extension}
Recall that we classified the subcategory of minimal objects in $\CX$ in Proposition \ref{prop-catminimal}. In this section we study the opposite extremal case, i.e.~we classify the {\em maximal} objects in $\CX$.  As in the minimal case, for every $\lambda\in X$ there is an up to isomorphism unique indecomposable maximal object $S_{\max}(\lambda)$ with highest weight $\lambda$. Using the results in the previous section we show that $S_{\max}(\lambda)$ is a finitely generated, hence free, $\SA$-module for all {\em dominant} weights $\lambda$. In particular, $\mathsf{R}(S_{\max}(\lambda))$ is an object in $\CO$ that admits a Weyl filtration.  It is plausible that $S_{\max}(\lambda)$ is free as an $\SA$-module for arbitrary $\lambda$ (but not of finite rank if $\lambda$ is not dominant), but we cannot prove this. 

In a subsequent section we show that $S_{\max}(\lambda)$ admits a {\em non-degenerate contravariant symmetric bilinear form} (cf.~Definition \ref{def-cform}) under the assumption that $\lambda$ is dominant (again, it is plausible that this restriction is not necessary). This fact is then used in the last section to show that $\mathsf{R}(S_{\max}(\lambda))$ is an indecomposble tilting module.

For the construction of the maximal objects in $\CX$  we need to assume that {\em projective covers} exist in the category of $\SA$-modules. Hence we assume that $\SA$ is a local ring. Here is a short reminder on projective covers.
 Let $\SR$ be a ring and $M$ an $\SR$-module. Recall that a {\em projective cover} of $M$ is a surjective homomorphism $\phi\colon P\to M$ such that $P$ is a projective $\SR$-module   and such that  any submodule $U\subset P$ with $\phi(U)=M$ satisfies $U=P$. If $\SR$ is a local ring, then projective covers exist for finitely generated $\SR$-modules. They can be constructed as follows. Denote by $\SF$ the residue field of $\SR$. For an $\SR$-module $N$ we let $\ol N=N\otimes_\SR\SF$ be the associated $\SF$-vector space. In the situation above, choose an isomorphism $\SF^{n}\cong \ol M$. This can be lifted to a homomorphism $\phi\colon \SR^{n}\to M$, and Nakayama's lemma implies that this is a projective cover.

\subsection{The maximal extension}
Let $I^\prime\subset I$ be a pair of closed subsets of $X$.  We now construct a maximal extension of an object $M^\prime$ in $\CX_{I^\prime}$, i.e.~an object $M$ in $\CX_I$ such that $M_{I^\prime}$ is isomorphic to $M^\prime$ and such that $M_{\lgl\mu\rgl}=M_{\{\mu\},\max}$ for all $\mu\in I\setminus I^\prime$. The following summarises the properties of this extension.

\begin{proposition} \label{prop-satextobj} Assume that $\SA$ is  local. 
 Let  $M^\prime$ be an object in $\CX_{I^\prime}$.  Then there exists an up to isomorphism unique  object $M$ in $\CX_I$ with the following properties.
\begin{enumerate}
\item $M_{I^\prime}$ is isomorphic to $M^\prime$.
\item An endomorphism $f\colon M\to M$ in $\CX_I$ is an automorphism if and only if $f_{I^\prime}\colon M_{I^{\prime}}\to M_{I^\prime}$ is an automorphism.
\item $M_{\lgl\mu\rgl}=M_{\{\mu\},\max}$ for all $\mu\in I\setminus I^\prime$. 
\end{enumerate}
\end{proposition}
\noindent
We call $M$ the {\em maximal extension} of $M^\prime$.

\begin{proof} As in the proof of Proposition \ref{prop-minimalextobj} we argue that it is sufficient to consider the case $I=I^\prime\cup\{\mu\}$ for some $\mu\not\in I^\prime$. So let us assume this in the course of the proof.
Let us prove that an object $M$ having the properties (1), (2), and (3) is unique. So suppose that $M_1$ and $M_2$ have these properties. Then, by (1), we have an isomorphism $M_{1I^\prime}\cong M_{2I^\prime}$.  From Proposition \ref{prop-mainext} we deduce that this isomorphism identifies $M_{1\{\mu\}}$ with $M_{2\{\mu\}}$ and hence $M_{1\{\mu\},\max}$ with $M_{2\{\mu\},\max}$, so $M_{1\lgl\mu\rgl}$ with $M_{2\lgl\mu\rgl}$ by property (3). So the condition in Lemma \ref{lemma-minimalextmor} is satisfied, so  the chosen isomorphism extends to a homomorphism $f\colon M_1\to M_2$. Reversing the roles of $M_2$ and $M_1$ yields a homomorphism $g\colon M_2\to M_1$ in an analogous way. Now property (2) implies that $g\circ f$ and $f\circ g$ are automorphisms. Hence $f$ and $g$ are isomorphisms.

It remains to show that an object $M$ with  properties (1), (2) and (3) exists.  First we consider the minimal extension $\widetilde M:=\mathsf{E}_{I^\prime}^I(M)$.  Then we can identify  $\tM_{\delta\mu}$ with $M^\prime_{\delta\mu}$. We set  $Q:=\tM_{\{\mu\},\max}/\tM_{\{\mu\}}$, so this is a torsion $\SA$-module. By property (X1), the $\SA$-module $\tM_{\delta\mu}$ is finitely generated, hence so is its submodule $\tM_{\{\mu\},\max}$ (recall that we always  assume that $\SA$ is Noetherian). Hence $Q$ is finitely generated.  Now we fix a projective cover  $\ol\gamma\colon D\to Q$ in the category of $\SA$-modules, and we denote by $\gamma\colon D\to \tM_{\{\mu\},\max}$ a lift of $\ol\gamma$. We can also  consider $\gamma$ as a homomorphism from $D$ to $\tM_{\delta\mu}$. 

We define $M$ as follows.  For all $\nu\in I^\prime$, $\alpha\in\Pi$ and $n>0$ we set $M_\nu:=\tM_\nu$, $E^M_{\nu,\alpha,n}:=E^{\tM}_{\nu,\alpha,n}$ and $F^M_{\nu,\alpha,n}:=F^{\tM}_{\nu,\alpha,n}$. Then we set $M_\mu:=\tM_\mu\oplus D$ and define $F^M_{\mu}:=(F^{\tM}_{\mu},0)^T\colon M_{\delta\mu}=\tM_{\delta\mu}\to M_{\mu}$ and $E^{M}_{\mu}:=(E^{\tM}_{\mu},\gamma)\colon M_{\mu}\to M_{\delta\mu}=\tM_{\delta\mu}$. 
We now show that $M=\bigoplus_{\nu\in I}M_\nu$ together with the $E$- and $F$-maps above is an object in $\CX_I$. Condition  (X1) is clearly satisfied. We now show that (X2) is also satisfied. Let $\nu\in I$, $\alpha,\beta\in\Pi$, $m,n>0$ and $v\in M_{\nu+n\beta}$. We need to show that 

\begin{align*}
E_{\alpha,m}F_{\beta,n}(v)=
\begin{cases}
F_{\beta,n}E_{\alpha,m}(v),&\text{ if $\alpha\ne\beta$}\\
\sum_{0\le r\le \min(m,n)} \qchoose{\lgl\nu,\alpha^\vee\rgl+m+n}{r}_{d_\alpha}F_{\alpha,n-r}E_{\alpha,m-r}(v), &\text{ if $\alpha=\beta$}.
\end{cases}
\end{align*}
If $\nu\ne\mu$, then this follows immediately from the fact that (X2) is satisfied for $\tM$, and in the case $\nu=\mu$ it follows as $E^M_{\mu,\alpha,m}$ coincides with $E^{\tM}_{\mu,\alpha,m}$ on the image of $F_{\mu,\beta,n}$.  
We now check the condition (X3). It  is satisfied for all $\nu\ne\mu$, as it is satisfied for $\tM$. In the case $\nu=\mu$, (X3a) follows from the corresponding condition for $\tM$ as the image of $F_\mu^M$ coincides with the image of $F_\mu^{\tM}$ and $E^M_\mu$ agrees with $E^{\tM}_\mu$ on this image. By construction $M_{\lgl\mu\rgl}=M_{\{\mu\},\max}$, so  the inclusion $M_{\{\mu\}}\subset M_{\lgl\mu\rgl}$ has a torsion cokernel. Finally, we have $\im F_\mu^{M}=(\tM_\mu,0)$, so the quotient $M_\mu/\im F^M_\mu$ is isomorphic to $D$. As $D$ was chosen to be a projective $\SA$-module, it is free. So we have indeed constructed an object in $\CX_I$.

 We need to check that $M$ satisfies the properties (1), (2)  and (3). Clearly, $M_{I^\prime}=\tM_{I^\prime}\cong M^\prime$ so (1) is satisfied.
We have already observed that $M_{\lgl\mu\rgl}=M_{\{\mu\},\max}$, hence (3).  Now let $f\colon M\to M$ be an endomorphism and suppose that $f_{I^\prime}\colon M_{I^\prime}\to M_{I^\prime}$ is an automorphism, i.e.~$f_\nu\colon M_\nu\to M_\nu$ is an automorphism for all $\nu\ne\mu$.  Then $f_{\delta\mu}\colon M_{\delta\mu}\to M_{\delta\mu}$ is an automorphism, and hence the restriction of $f_{\mu}|_{\im F_\mu}\colon \im F_\mu\to \im F_\mu$  is an automorphism.  Applying $E_\mu$ shows that $f_{\{\mu\}}$ is an automorphism of $M_{\{\mu\}}$. Hence $f_{\delta\mu}$ induces an automorphism of $M_{\{\mu\},\max}$ and we obtain an induced automorphism of the quotient $Q$ defined earlier in this proof. As $\ol\gamma\colon D\to Q$ is a projective cover, also the induced endomorphism on $D$  must be an automorphism. Hence $f_\mu$ is an automorphism, and hence so is $f$. Hence (2) also holds.
\end{proof}

  
  \subsection{The first example, continued} \label{subsec-excont} Recall the setup of the example discussed in Section \ref{ex-first}. In the first step we found that the homomorphism $E_\mu\colon \SA\to\SA$ is multiplication with $[\lambda]$. So $M_{\{\mu\}}=[\lambda]\SA\subset M_{\delta\mu}=\SA$. Suppose that $[\lambda]$ does not vanish in $\SA$. Then  $M_{\{\mu\}}=[\lambda]\SA\subset M_{\{\mu\},\max}=\SA$. If $[\lambda]$ is not invertible in $\SA$, then we have a strict embedding and a torsion quotient.  In the notation of the above proposition, we need to choose a projective cover $\gamma\colon D\to M_{\{\mu\},\max}/M_{\{\mu\}}$ and obtain $M_\mu:=\SA\oplus D$ as the  maximal extension. 

\subsection{An example in the $A_2$-case} Let $R=\{\pm\alpha,\pm\beta,\pm(\alpha+\beta)\}$ be the root system of type $A_2$ with $\Pi=\{\alpha,\beta\}$. Fix $\lambda\in X$ and set $\mu:=\lambda-\alpha-\beta$. We set $I^\prime:=\{\nu\in X\mid \nu>\mu\}$ and $I:=I^\prime\cup\{\mu\}$. Then we define $M^\prime:=\bigoplus_{\nu\in I^\prime} M^\prime_\nu$ by $M^\prime_{\lambda-\alpha}=M^\prime_{\lambda-\beta}=M^\prime_{\lambda}=\SA$, and  all other weight spaces are $\{0\}$. Let all  $F$-maps between the non-zero weight spaces be  the identity. Then the commutation relations force us to define $E_{\alpha,1}\colon M^\prime_{\lambda-\alpha}\to M^\prime_\lambda$ as multiplication with $[\lgl\lambda,\alpha^\vee\rgl]$ and $E_{\beta,1}\colon M^\prime_{\lambda-\beta}\to M^\prime_\lambda$ as multiplication with $[\lgl\lambda,\beta^\vee\rgl]$. All other $E$-maps are zero, of course. We assume that $[\lgl\lambda,\alpha^\vee\rgl]$ and $[\lgl\lambda,\beta^\vee\rgl]$ do not vanish in $\SA$. 

Let $\widetilde M$ be the minimal extension of $M^\prime$ to the weight $\mu$. We then have $\widetilde M_{\delta\mu}=M_{\mu+\alpha}\oplus M_{\mu+\beta}=\SA\oplus\SA$ with basis $(F_{\alpha,1}v,F_{\beta,1}v)$, where $v$ is a generator of $M^\prime_\lambda$, and $\widetilde M_\mu=\SA\oplus\SA$ with basis $(F_{\alpha,1}F_{\beta,1}v,F_{\beta,1}F_{\alpha,1} v)$. We calculate
\begin{align*}
E_{\alpha,1}F_{\alpha,1}F_{\beta,1}v&=F_{\alpha,1}E_{\alpha,1}F_{\beta,1}v+[\lgl\lambda-\beta,\alpha^\vee\rgl]F_{\beta,1}v \\
&= [\lgl\lambda-\beta,\alpha^\vee\rgl]F_{\beta,1}v,\\
E_{\alpha,1}F_{\beta,1}F_{\alpha,1}v&=F_{\beta,1}E_{\alpha,1}F_{\alpha,1}v = [\lgl\lambda,\alpha^\vee\rgl] F_{\beta,1}v,\\
\end{align*}
and (symmetrically) for $E_{\beta,1}$. The homomorphism  $E_\mu$ is hence given by the matrix
$
\left(\begin{matrix}[\lgl\lambda,\alpha^\vee\rgl] & [\lgl\lambda-\alpha,\beta^\vee\rgl]\\ [\lgl\lambda-\beta,\alpha^\vee\rgl]&[\lgl\lambda,\beta^\vee\rgl]\end{matrix}\right)$. 

Let us specialize this further. Suppose  $\lambda=\rho$ and $p=3$, and let $\SA=\SZ_\fp$ be the localization of $\DZ[v,v^{-1}]$ at the kernel of the homomorphism $\DZ[v,v^{-1}]\to\DF_3$ that sends $v$ to $1$.  The matrix then is $
\left(\begin{matrix}[1] & [2]\\ [2]&[1]\end{matrix}\right)$. The element $[2]$ is invertible in $\SZ_\fp$, and the matrix is congruent to 
$$
\left(\begin{matrix}  [1]-[2]^2&[2]\\ 0&[1]\end{matrix}\right)=\left(\begin{matrix}  1-(v+v^{-1})^2 &v+v^{-1}\\ 0& 1\end{matrix}\right)=\left(\begin{matrix}  -v^2-1-v^{-2}  &v+v^{-1}\\ 0& 1\end{matrix}\right).
$$
 So its image has $v^2+1+v^{-2}=v^{-2}(v^2+v+1)(v^2-v+1)$-torsion. Since $v^{-2}(v^2-v+1)$ is invertible in $\SZ_\fp$, this is $v^2+v+1$-torsion. (Note that $v^2+v+1$ is the third cyclotomic polynomial.) Hence in this situation the $\SA$-module $M_{\rho-\alpha-\beta}=M_0$ is free of rank $3$.

 \subsection{The category of maximal objects}
Here is an analogue of Proposition \ref{prop-catminimal} in the maximal case.

\begin{proposition} \label{prop-catsat} Suppose that $\SA$ is local. 
\begin{enumerate}
\item  For all $\lambda\in X$ there exists an up to isomorphism unique object $S_{\max}(\lambda)$ in $\CX$ with the following properties.
\begin{enumerate}
\item $S_{\max}(\lambda)_\lambda$ is free of rank $1$ and $S_{\max}(\lambda)_\mu\ne\{0\}$ implies $\mu\le\lambda$. 
\item $S_{\max}(\lambda)$ is indecomposable and maximal. 
\end{enumerate}
\end{enumerate}
Moreover, the objects $S_{\max}(\lambda)$ characterized in (1) have the following properties.
\begin{enumerate}\setcounter{enumi}{1}
\item The endomorphism ring $\End_{\CX}(S_{\max}(\lambda))$ is local for all $\lambda\in X$.
\item Let $S$ be a  maximal object in $\CX$. Then there is an index set $J$ and some elements $\lambda_i\in X$ for $i\in J$ such that $S\cong \bigoplus_{i\in J}S_{\max}(\lambda_i)$. 
\end{enumerate}
\end{proposition}

\noindent
Sometimes we will write  $S_{\max,\SA}(\lambda)$ to incorporate the ground ring.

\begin{proof} The proof follows closely the proof of the analogous proposition in the minimal case (Proposition \ref{prop-catminimal}). So we start with proving that there exists an object $S_{\max}(\lambda)$ satisfying  (1a), (1b) and (2) for all $\lambda\in X$. So fix $\lambda$ and     set ${I_\lambda}:=\{\mu\in X\mid \lambda\le\mu\}$. Let  $S^\prime$ be the object  in $\CX_{I_\lambda}$ constructed in Proposition \ref{prop-catminimal} and denote by $S_{\max}(\lambda)$ the maximal extension provided by Proposition \ref{prop-satextobj}. Then the construction of  $S_{\max}(\lambda)$ implies that its support is contained in $\{\mu\in X\mid \mu\le\lambda\}$, and $S_{\max}(\lambda)_\lambda=S^\prime_\lambda=\SA$. So (1a) is satisfied.  Proposition  \ref{prop-satextobj} also implies that an endomorphism $f$ of $S_{\max}(\lambda)$ is an automorphism if and only if its restriction to $S_{\max}(\lambda)_{I_\lambda}\cong S^\prime$ is an automorphism. As $\End_{\CX_{I_\lambda}}(S^\prime)=\SA\cdot\id$, the endomorphism ring of $S_{\max}(\lambda)$ is local. Hence (2) holds and $S_{\max}(\lambda)$ is indecomposable. Finally, Proposition \ref{prop-satextobj} implies that $S_{\max}(\lambda)_{\{\mu\}}=S_{\max}(\lambda)_{\lgl\mu\rgl}$ for all $\mu\in X\setminus I_\lambda$. As the same identity also holds for all $\mu\in I_\lambda$ (as it holds for $S^\prime$), we deduce that $S_{\max}(\lambda)$ is maximal. Hence (1b) is also satisfied. 

We are now left with proving  property (3) with the  $S_{\max}(\lambda)$ being  the objects  constructed above.  Let $S$ be a maximal object and let $\lambda$ be a maximal weight of $S$. As in the proof of Proposition \ref{prop-catminimal} we can fix an isomorphism $S_{I_\lambda}\cong (S^\prime)^{\oplus n}\cong \tilde S_{I_\lambda}$ with $\tilde S=S_{\max}(\lambda)^{\oplus n}$. Since  $S$ and  $\tilde S$ are both maximal,   Lemma \ref{lemma-minimalextmor} implies  that this isomorphism extends, so there are morphisms $f\colon  \tilde S\to S$ and $g\colon S\to  \tilde S$ such that $(g\circ f)|_{I_\lambda}$ is the identity. Using  the already proven property (2) for $S_{\max}(\lambda)$ we deduce that  $g\circ f$ is an automorphism. Hence $\tilde S=S_{\max}(\lambda)^{\oplus n}$ is isomorphic to a direct summand of $S$. By construction, $\lambda$ is not a weight of a direct complement. From here we can continue by induction to prove (3). 
\end{proof}

Note that the main ingredients in the existence result  above are Propositions \ref{prop-minextfun} and \ref{prop-satextobj}. The proofs of these  propositions are constructive, i.e.~they can be read as an algorithm to construct the weight spaces of the objects $S_{_{\max}}(\lambda)$ inductively, starting with the highest weight space.

\subsection{Objects generated by dominant weights}\label{sec-domweights}

Let $M$ be an object in $\CX_\SA$. Denote by $\tM\subset M$ the smallest  $X$-graded $\SA$-submodule that is stable under all $E$- and $F$-maps and contains $M_\nu$ for all dominant weights $\nu$. We say that $M$ is {\em generated by dominant weights} if $\tM=M$.

\begin{lemma}\label{lemma-gendom} Suppose that $\SA$ is generic and that $M$ is generated by dominant weights. Then $M$ is contained in $\CX^{\fin}$. In particular,  $\mathsf{R}(M)$ admits a Weyl filtration.
\end{lemma}
\begin{proof} Let $\SK$ be the quotient field of $\SA$. As $M$ is generated by dominant weights, so is $M_\SK\in\CX_\SK$. By Lemma \ref{lemma-fieldcase}, $M_\SK$ splits into a direct sum of various $S_\SK(\nu_i)$ with dominant highest weights $\nu_i$. By (X1), the number of direct summands is finite. Now $S_\SK(\nu_i)\cong \mathsf{S}(L_\SK(\nu_i))$ by Proposition \ref{prop-irreps}. As $\nu_i$ is dominant, each $L_\SK(\nu_i)$ is finite dimensional, hence so is each $S_\SK(\nu_i)$.  Hence $M_\SK$ is a finite dimensional vector space. In particular, its set of weights is finite. This equals the set of weights of $M$, so by the remark following Definition \ref{def-Xfin}, $M$ is finitely generated as an $\SA$-module, hence is contained in $\CX^{\fin}$. The last statement follows from Theorem \ref{thm-Wflag}. 
\end{proof}

\subsection{Minimal and maximal extensions and finite rank}
Suppose that $\SA$ is generic and local. 
\begin{proposition} \label{prop-domfin} For all $\lambda\in X$, the following statements are equivalent.
\begin{enumerate}
\item $\lambda$ is dominant.
\item $S_{\min}(\lambda)$ is a finitely generated $\SA$-module.
\item $S_{\max}(\lambda)$ is a finitely generated $\SA$-module.
\end{enumerate}
\end{proposition}
It follows from Theorem \ref{thm-Wflag} that $S_{\min}(\lambda)$ and $S_{\max}(\lambda)$ (for dominant $\lambda$) yield, under the functor $\mathsf{R}$, objects in $\CO$ that admit a Weyl filtration. In particular, they are {\em free} of finite rank as $\SA$-modules. 
\begin{proof} Again we denote by $\SK$ the quotient field of $\SA$. Now $S_{\min}(\lambda)$ is finitely generated as an $\SA$-module if and only if its set of weights is finite, which is the case if and only if $S_{\min}(\lambda)_\SK$ is a finite dimensional $\SK$-vector space. Under the functor $\mathsf{R}$ the latter corresponds to $L_\SK(\lambda)$ which is finite dimensional if and only if $\lambda$ is dominant. Hence (1) and (2) are equivalent. If (3) holds, then $\mathsf{R}(S_{\max}(\lambda)_\SK)$ is a finite dimensional object in $\CO_\SK$. Hence every simple subquotient has a dominant highest weight. In particular, its maximal weight $\lambda$ is dominant. Hence (3) implies (1). 

We are left with proving that $S:=S_{\max}(\lambda)$ is finitely generated as an $\SA$-module for dominant $\lambda$. So fix a dominant weight $\lambda$. We show that $S$ is generated by dominant weights in the sense defined in Section \ref{sec-domweights}. Then our claim follows from Lemma \ref{lemma-gendom}. So let $\tS\subset S$ be the minimal $E$- and $F$-stable graded subspace that contains all $S_\nu$ for dominant weights $\nu$. If $\tS\ne S$, then there exists a weight $\mu$ with $\tS_\mu\ne S_\mu$. Let us choose a maximal such weight $\mu$. This $\mu$ cannot be dominant, but $\mu<\lambda$.

Now let $I:=\{\nu\in X\mid \mu<\nu\}\cup\{\nu\in X\mid \nu\not\le\lambda\}$. This is a closed subset of $X$ and it does not contain $\mu$. Moreover,   $X\setminus I\subset\{\le \lambda\}$ is  bounded from above. Now set $C:=\mathsf{E}_{I}^X(S_{\max}(\lambda)_I)$. By the maximality of $\mu$, $S_{\max}(\lambda)_I$ and hence $C$ is generated by dominant weights.  So Lemma \ref{lemma-gendom} implies that $\mathsf{R}(C)$ admits a Weyl filtration. In particular,  $\im F^C_\mu$ and $C_{\delta\mu}$ are free $\SA$-modules (of finite rank). 
Now $\tS_\mu\ne S_\mu$ implies, by the construction of $S=S_{\max}(\lambda)$, that  the cokernel of  $E^S_\mu|_{\im F^S_\mu}\colon \im F^S_\mu\to S_{\delta\mu}$  has  non-vanishing torsion. But this homomorphism coincides with $E_\mu^C|_{\im F_\mu^C}\colon \im F_\mu^C\to C_{\delta\mu}$ as $C_I=S_I$. This implies that $E_\mu^C|_{\im F_\mu^C}\colon \im F_\mu^C\to C_{\delta\mu}$ is not injective over the residue field $\SF$ of $\SA$. From this we deduce that there exists a primitive vector of weight $\mu$ in the $U_\SF$-module $\ol{\mathsf{R}(C)}$. As the $U_\SA$-module $\mathsf{R}(C)$ admits a Weyl filtration, there must be a Weyl module having a primitive vector of weight $\mu$ when reduced to $\SF$. But this implies that $\mu$ is dominant, which contradicts our assumption. 
\end{proof}

\section{Contravariant forms}  
In order to connect maximal objects in $\CX$ to tilting modules in $\CO$, we need to add another ingredient to the theory: contravariant forms. The main result of this section is that there exists a {\em non-degenerate} contravariant form on the maximal object $S_{\max}(\lambda)$, provided that each weight space $S_{\max}(\lambda)_\mu$ is a free $\SA$-module. Using the results of the previous section, we know that this is the case if $\lambda$ is a dominant weight (probably this restriction is not necessary). In the next section we show that the existence of a non-degenerate contravariant form on $S_{\max}(\lambda)$ implies that $\mathsf{R}(S_{\max}(\lambda))$ is a tilting module.  

\subsection{Contravariant forms}
Let $\SA$ be a Noetherian unital $\SZ$-algebra that is a domain. Let $I$ be a closed subset of $X$ and $M$  an object in $\CX_{I,\SA}$. Let $b\colon M\times M\to \SA$ be an $\SA$-bilinear form.
\begin{definition} \label{def-cform} We say that $b$ is a {\em symmetric contravariant form on $M$} if the following holds.
\begin{enumerate}
\item  $b$ is symmetric.
\item $b(m,n)=0$ if $m\in M_\mu$ and $n\in M_\nu$ and $\mu\ne\nu$. 
\item  For all $\mu\in I$, $\alpha\in\Pi$, $n>0$, $x\in M_{\mu}$, $y\in M_{\mu+n\alpha}$ we have
$$
b(E_{\mu,\alpha,n}(x),y)=b(x,F_{\mu,\alpha,n}(y)).
$$
\end{enumerate}
\end{definition}
Property (2) implies that every symmetric contravariant form $b$ splits into the direct sum $\bigoplus_{\mu\in I} b_\mu$ of its weight components. We write $b_{\delta\mu}$ for the restriction of $b$ to $M_{\delta\mu}\times M_{\delta\mu}$. Note that if  $b$ satisfies (1) and (2), then property (3) is equivalent to $b_\mu(x,F_\mu(y))=b_{\delta\mu}(E_\mu(x),y)$ for all $\mu\in I$, $x\in M_\mu$, $y\in M_{\delta\mu}$.

\subsection{Non-degeneracy and torsion vanishing}

Recall that a symmetric bilinear form $b\colon N\times N\to\SA$ on a finitely generated  $\SA$-module $N$   is called {\em non-degenerate} if the induced homomorphism $N\to N^\ast:=\Hom_\SA(N,\SA)$, $n\mapsto b(n,\cdot)$,  is an isomorphism.  If $\SA$ is a local algebra with residue field $\SF$, and $N$ is a free $\SA$-module,  then $b$ is non-degenerate if and only if its specialization $\ol b$ on $\ol N:=N\otimes_\SA\SF$ is non-degenerate. 

A symmetric contravariant bilinear form $b$ on an object $M$ in $\CX_I$ is non-degenerate if and only if the weight components $b_\mu$ are non-degenerate for all $\mu\in I$. 
We denote by $\rad\,  b=\{n\in M\mid b(n,m)=0\text{ for all $m\in M$}\}$ the radical of $b$. Note that $\rad\,  b=0$ if $b$ is non-degenerate, but the converse statement is not true in general. If $\SA=\SK$ is a field, and each $M_\mu$ is finite dimensional, then non-degeneracy is equivalent to the vanishing of the radical.

In order to study contravariant forms, the following quite general result will be helpful for us. It holds for any commutative ring $\SA$.

\begin{lemma}\label{lemma-biforms} Let $S$ and $T$ be $\SA$-modules and assume that $T$ is projective as an $\SA$-module.  Let $F\colon S\to T$ and $E\colon T\to S$ be homomorphisms. Suppose that 
 $b_S\colon S\times S\to \SA$ and $b_T\colon T\times T\to \SA$ are symmetric, non-degenerate bilinear forms such that  
$$
b_T(F(s),t)=b_S(s,E(t))
$$
for all $s\in S$ and $t\in T$. 
If the inclusion $i_F\colon F(S)\subset T$ splits, then  the inclusion $i_E\colon E(T)\subset S$ splits as well.
\end{lemma}

\begin{proof} For $s\in \ker F$ we have $0=b_T(F(s),t)=b_S(s,E(t))=0$ for all $t\in T$.  Hence we can define a homomorphism $\phi\colon E(T)\to F(S)^\ast$ by setting $\phi(E(t))(F(s))=b_S(s,E(t))$ (i.e., this definition doesn't depend on the choice of $s$ in the preimage of $F(s)$). Then the right hand side of the diagram

\centerline{
\xymatrix{
T\ar[d]_{b_T}\ar[r]^E&E(T)\ar[r]^{i_E}\ar[d]^{\phi}&S\ar[d]^{b_S}\\
T^\ast\ar[r]^{i_F^\ast}&F(S)^\ast\ar[r]^{F^\ast}&S^\ast
}
}
\noindent
commutes. By the adjointness property, $\phi(E(t))(F(s))=b_T(F(s),t)$ for all $s\in S$ and $t\in T$, hence also the left hand side commutes. The vertical homomorphisms on the left and the right are isomorphisms (as $b_S$ and $b_T$ are supposed to be non-degenerate).
 As $i_E$ is injective, $\phi$ is injective. Suppose that   $i_F\colon F(S)\subset T$ splits. Then the dual homomorphism   $i_F^\ast\colon T^\ast\to F(S)^\ast$ is surjective. Hence $\phi$ is surjective, so it is an isomorphism. As $F(S)$ is projective (it is a direct summand of $T$), the surjective homomorphism $F\colon S\to F(S)$ splits, and hence $F^\ast\colon F(S)^\ast\to S^\ast$ splits. Hence the inclusion $i_E\colon E(T)\to S$ splits. \end{proof}

The following is our main application of Lemma \ref{lemma-biforms}. Let $I$ be a closed subset of $X$. We assume that $\SA$ is a unital Noetherian $\SZ$-domain.

\begin{proposition}\label{prop-cform}  Let $M$ be an object in $\CX_I$ that is free as an $\SA$-module. Suppose that there exists a non-degenerate symmetric contravariant form on $M$. Then $M$ is maximal.  \end{proposition}

\begin{proof} We need to show that $M_{\lgl\mu\rgl}=M_{\{\mu\},\max}$ for all $\mu\in X$. This is equivalent to $M_{\delta\mu}/M_{\lgl\mu\rgl}$ being torsion free for any $\mu$. In the following we will show that  the inclusion $M_{\lgl\mu\rgl}\subset M_{\delta\mu}$ splits for any $\mu$, so the torsion freeness of $M_{\delta\mu}$ implies that  the quotient $M_{\delta\mu}/M_{\lgl\mu\rgl}$ is torsion free as well. 

Let   $b\colon M\times M\to \SA$ be a non-degenerate symmetric contravariant form on $M$. For all $\mu\in X$ it  induces  symmetric, non-degenerate bilinear forms $b_{\delta\mu}$ and $b_\mu$ on the $\SA$-modules $M_{\delta\mu}$ and $M_\mu$, resp., with
$$
b_\mu(F_\mu(v),w)=b_{\delta\mu}(v,E_\mu(w))
$$
for all $v\in M_{\delta\mu}$ and $w\in M_\mu$. 
By assumption,  $M_\mu$ and $M_{\delta\mu}$ are  free $\SA$-modules of finite rank. Moreover, by axiom (X3), the quotient $M_\mu/F_\mu(M_{\delta\mu})$ is a free $\SA$-module. Hence the inclusion $F_\mu(M_{\delta\mu})\subset M_\mu$ splits and we can apply Lemma \ref{lemma-biforms} and deduce that the inclusion $M_{\lgl\mu\rgl}=E_\mu(M_\mu)\subset M_{\delta\mu}$ splits. \end{proof}
  
\subsection{Extension of contravariant forms} Let $I^\prime\subset I$ be closed subsets of $X$. Suppose that $M$ is an object in $\CX_I$ and $b\colon M\times M\to \SA$ is a contravariant form on $M$. Then the restriction $b_{I^\prime}$ of $b$ to $M_{I^\prime}\times M_{I^\prime}$ is a contravariant form on $M_{I^\prime}$. We now show that contravariant forms extend uniquely to minimal extensions of objects.

\begin{lemma}\label{lemma-extbiform} Let $I^\prime\subset I$ be a pair of closed subsets of $X$. Let $M^\prime$ be an object in $\CX_{I^\prime}$ and let $b^\prime$ be a symmetric contravariant form on $M^\prime$. Let $M:=\mathsf{E}_{I^\prime}^I M^\prime$ be the minimal extension, and fix an isomorphism $M_{I^\prime}\cong M^\prime$. Then there exists a unique symmetric contravariant form $b$ on $M$ that restricts to $b^\prime$ under the chosen isomorphism. Moreover, the following holds: 

\begin{enumerate}
\item If $\rad\, b^\prime=0$, then $\rad\,  b=0$.
\item If $\SA=\SK$ is a field and $b^\prime$ is non-degenerate, then $b$ is non-degenerate as well.
\end{enumerate}
 \end{lemma}

\begin{proof}  Again we can assume that $I=I^\prime\cup\{\mu\}$ for some $\mu\not\in I^\prime$ by an inductive argument. As in the proof of Proposition 
\ref{prop-minimalextobj} we construct the $\SA$-module $\widehat M_\mu$ together with the homomorphisms $\widehat F_\mu\colon M_{\delta\mu}\to\widehat M_\mu$ and $\widehat E_\mu\colon \widehat M_\mu\to M_{\delta\mu}$. Then, by construction, $M_\mu=\widehat M_\mu/\ker \widehat E_\mu$, and $E_\mu$ and $F_\mu$ are the homomorphisms induced by $\widehat E_\mu$ and $\widehat F_\mu$. Recall that $\widehat F_\mu$ is the direct sum of the homomorphisms $\widehat F_{\alpha,m}\colon M_{\mu+m\alpha}\to \widehat M_\mu$ and that it is an isomorphism. So we can define a bilinear form $\widehat b_\mu$ on $\widehat M_\mu$ by setting
$$
\widehat b_\mu(\widehat F_{\alpha,m}(x),\widehat F_{\beta,n}(y)):=b^\prime(x,\widehat E_{\alpha,m}\widehat F_{\beta,n}(y))
$$
for $\alpha,\beta\in\Pi$, $m,n>0$, $x\in M_{\mu+m\alpha}$ and $y\in M_{\mu+n\beta}$.

Let us show now that this bilinear form is symmetric. This amounts to showing that $b^\prime(x,\widehat E_{\alpha,m}\widehat F_{\beta,n}(y))=b^\prime(y,\widehat E_{\beta,n}\widehat F_{\alpha,m}(x))$. We now use the definition of the homomorphism $\widehat E_{\beta,n}$ that the reader finds in the proof of Proposition \ref{prop-minimalextobj}. Suppose that $\alpha\ne\beta$. Then
\begin{align*}
b^\prime(x,\widehat E_{\alpha,m}\widehat F_{\beta,n}(y))&=b^\prime(x,F_{\beta,n}E_{\alpha,m}(y))\quad\text{(by definition of $\widehat E_{\alpha,m}$)}\\
&=b^\prime(F_{\alpha,m}E_{\beta,n}(x),y)\quad\text{(by contravariance of $b^\prime$)}\\
&=b^\prime(\widehat E_{\beta,n}\widehat F_{\alpha,m}(x),y)\quad\text{(by definition of $\widehat E_{\beta,n}$)}\\
&=b^\prime(y,\widehat E_{\beta,n}\widehat F_{\alpha,m}(x))\quad\text{(as $b^\prime$ is symmetric)}.
\end{align*}
Now suppose that $\alpha=\beta$.  Then we can write $\widehat E_{\alpha,m}\widehat F_{\alpha,n}(y)=\sum_{r\ge 0} c_r F_{\alpha,n-r}E_{\alpha,m-r}(y)$ with $c_r=\qchoose{\lgl\mu,\alpha^\vee\rgl+m+n}{r}_{d_\alpha}$. It will be important later that  $c_r$ depends only on $\mu$, $\alpha$, $r$  and the sum $m+n$. The contravariance of $b^\prime$ then yields
\begin{align*}
b^\prime(x,\widehat E_{\alpha,m}\widehat F_{\alpha,n}(y))&=b^\prime(x,\sum_{r\ge 0} c_r F_{\alpha,n-r}E_{\alpha,m-r}(y))\\
&=b^\prime(\sum_{r\ge 0} c_r F_{\alpha,m-r}E_{\alpha,n-r}(x),y).
\end{align*}
Now $\sum_{r\ge 0} c_r F_{\alpha,m-r}E_{\alpha,n-r}(x)=\widehat E_{\alpha,n}\widehat F_{\alpha,m}(x)$, again by definition (note that the coefficient $c_r$ is the same if we replace the triple $(y,m,n)$ with $(x,n,m)$!), so 
$$
b^\prime(x,\widehat E_{\alpha,m}\widehat F_{\alpha,n}(y))=b^\prime(\widehat E_{\alpha,n}\widehat F_{\alpha,m}(x),y).
$$
Using the symmetry of $b^\prime$ one last time yields the required result. So $\widehat b_\mu$ is symmetric. 

We now need to show that $\widehat b_\mu\colon \widehat M_\mu\times \widehat M_\mu\to\SA$ decends to a bilinear form on $M_\mu=\widehat M_\mu/\ker \widehat E_\mu$. 
So suppose that $y\in \ker \widehat E_\mu$. Then $\widehat b_\mu(\widehat F_{\mu}(x),y)=b^\prime(x,\widehat E_\mu(y))=0$ for all $x\in M_{\delta\mu}$.  As $\widehat F_\mu$ is surjective,  $y$ is contained in the radical of $\widehat b_\mu$, and hence we obtain an induced bilinear form $b_\mu$ on the quotient $M_\mu=\widehat M_\mu/\ker \widehat E_\mu$. We now define the bilinear form $b$ on $M=M^\prime\oplus M_\mu$ as the (orthogonal) direct sum of  $b^\prime$ and $b_\mu$. 

Now we show that $b$ is contravariant. We need to check that 
$$
b(F_{\nu,\alpha,m}(x),y)=b(x,E_{\nu,\alpha,m}(y))
$$
for all $\nu\in I$, $\alpha\in\Pi$, $m>0$, $x\in M_{\nu+m\alpha}$, $y\in M_\nu$. For $\nu\ne\mu$ this follows directly from the contravariance of $b^\prime$. In the case $\mu=\nu$ this is a direct consequence of the definition of $\widehat b_\mu$ and the fact that $E_{\mu,\alpha,m}$ and $F_{\mu,\alpha,m}$ are induced by $\widehat E_{\mu,\alpha,m}$ and $\widehat F_{\mu,\alpha,m}$.

We now show that the form $b$ is unique. But note that the contravariance forces us to have 
$b_\mu(F_{\alpha,m}(x), F_{\beta,n}(y))=b^\prime(x, E_{\alpha,m} F_{\beta,n}(y))$. As the homomorphism $F_\mu\colon M_{\delta\mu}\to M_\mu$ is surjective (as $M$ is the minimal extension), this shows that there is at most one extension of $b^\prime$ to a contravariant form on $M$. 

Now we prove (1). Suppose that the radical of $b^\prime$ vanishes. Suppose that $x$ is in the radical of $b$. Then $x\in M_\mu$. The equation $0=b_\mu(x,F_\mu(y))=b^\prime_{\delta\mu}(E_\mu(x),y)$ shows that $E_\mu(x)$ is in the radical of $b^\prime_{\delta\mu}$. The non-degeneracy of $b^\prime_{\delta\mu}$ implies that $E_\mu(x)=0$. As $F_\mu\colon M_{\delta\mu}\to M_\mu$ is surjective and as $E_\mu$ is injective on the image of $F_\mu$ we deduce that $x=0$. Hence $b_\mu$, and hence $b$, has vanishing radical.

Now note that if $\SA=\SK$ is a field, then a symmetric bilinear form on a finite dimensional $\SK$-vector space is non-degenerate if and only if its radical vanishes. Hence (1) implies that $b_\mu$ is non-degenerate if $b^\prime$ is non-degenerate. So we deduce (2).
\end{proof}

\subsection{Non-degenerate extensions} Suppose that we are in the situation of Lemma \ref{lemma-extbiform} (with $I=I^\prime\cup\{\mu\}$) and assume that the form $b^\prime$ on $M^\prime$ is non-degenerate. Its unique extension $b$ on $\mathsf{E}_{I^\prime}^I(M^\prime)$  possibly is degenerate. The next result shows that in this case we can extend $\mathsf{E}_{I^\prime}^I(M^\prime)$ further in such a way that $b^\prime$ admits a non-degenerate contravariant extension $b$. However, we have to assume that $\mathsf{E}_{I^\prime}^I(M^\prime)_\mu$ is a free $\SA$-module and that $\SA$ is local. 

\begin{proposition}\label{prop-selfdualobj} Suppose that $\SA$ is local. Let $I^\prime$ be a closed subset of $X$ and let $\mu\in X$ be such that $I:=I^\prime\cup\{\mu\}$ is also closed. Let $M^\prime$ be an object in $\CX_{I^\prime}$ and $b^\prime$  a non-degenerate symmetric contravariant bilinear form on $M^\prime$. Suppose that the $\SA$-module $\mathsf{E}_{I^\prime}^I(M^\prime)_\mu$ is  free.
Then there exists an object $M$ in $\CX_{I}$ and a non-degenerate symmetric contravariant form $b$ on $M$ with the following properties.
\begin{enumerate}
\item  There is an isomorphism $M_{I^\prime}\cong M^\prime$ that identifies $b_{I^\prime}$ with $b^\prime$. 

\item An endomorphism $f$ of $M$ is an automorphism if and only if its restriction to $M_{I^\prime}$ is an automorphism. 
\end{enumerate}
\end{proposition}
\begin{proof} Denote by $\widetilde b$ the unique symmetric contravariant form on $\widetilde M:=\mathsf{E}_{I^\prime}^I (M^\prime)$ that extends $b^\prime$ (Lemma \ref{lemma-extbiform}).  If $\widetilde b$ is non-degenerate, we set $M:=\widetilde M$ and $b:=\widetilde b$. Then property (1) holds by construction, and property (2) follows from Proposition \ref{prop-minextfun}.  

Now suppose that $\widetilde b$ is degenerate. Let $\SF$ be the residue field of $\SA$ and consider the bilinear form  $\ol{\tilde b}$ on $\ol{\widetilde M}:=\widetilde M\otimes_\SA\SF$. This is a symmetric bilinear contravariant form on    $\ol{\widetilde M}$ and it is degenerate. As $b^\prime$ is non-degenerate, the radical $\ol R$ of $\ol{\tilde b}$ is contained in $\ol{\widetilde M}_\mu$. Now let us fix a vector space complement to $\ol R$ in $\ol{\widetilde M}_\mu$, i.e.~a decomposition $\ol{\widetilde M}_\mu=\ol D\oplus\ol R$. Then the restriction of $\ol{\tilde b}$ to $\ol D\times \ol D$ is non-degenerate. By assumption,  $\widetilde M_\mu$ is a free $\SA$-module, so we can choose a lift $\widetilde M_\mu=D\oplus R$ of the former decomposition with free $\SA$-modules $D$ and $R$ (of finite rank, as $\widetilde M_\mu$ is of finite rank). Now let $S$ be a free $\SA$-module of the same rank as $R$, and  fix  a non-degenerate bilinear pairing $c\colon  R\times S\to\SA$.  Let $b_\mu$ be the bilinear form on $\widetilde M_\mu\oplus S=D\oplus R\oplus S$ defined by
$$
b_\mu(d_1+r_1+s_1,d_2+r_2+s_2)=\widetilde b_\mu(d_1+r_1,d_2+r_2)+c(r_1,s_2)+c(r_2,s_1),
$$
where $d_i\in D$, $r_i\in R$, $s_i\in S$ for $i=1,2$. 
This is a symmetric bilinear form. We claim that it is non-degenerate. Let $\ol b_\mu$ be the specialization of $b_\mu$ to $\ol D\oplus \ol R\oplus \ol S$. It agrees with $\ol{\tilde b}_\mu$ on $\ol D\oplus\ol R$. Recall that $\ol R$ is the radical of $\ol{ \tilde b}$, so $\ol b_\mu(r,d)=0$ for $r\in\ol R$ and $d\in\ol D$. Hence we have
$$
\ol b_\mu(d_1+r_1+s_1,d_2+r_2+s_2)=\ol{\tilde b}_\mu(d_1,d_2)+\ol c(r_1,s_2)+\ol c(r_2,s_2)
$$
for $d_i\in \ol D$, $r_i\in \ol R$, $s_i\in\ol S$ for $i=1,2$. 
Hence $\ol b_\mu$ splits into the non-degenerate bilinear form $\ol{\tilde b}_\mu$ on $\ol D\times \ol D$ and the non-degenerate bilinear form on $(\ol R\oplus \ol S)\times (\ol R\oplus \ol S)$ induced by the non-degenerate pairing $\ol c\colon \ol R\times \ol S\to \SF$. Hence $\ol b_\mu$, and hence $b_\mu$ is non-degenerate.

Now set $M:={\widetilde M}\oplus S$ and let  $M_\mu:={\widetilde M}_\mu\oplus S$ be its $\mu$-component. Define $F_\mu\colon M_{\delta\mu}\to M_\mu$ as the composition of $F_\mu^{\widetilde M}\colon {\widetilde M}_{\delta\mu}=M_{\delta\mu}\to \widetilde M_\mu$ and the inclusion $\widetilde M_\mu\subset M_\mu$ of the direct summand.  Define a homomorphism $E^\prime_\mu\colon S\to M_{\delta\mu}$ in the following way. For $s\in S$, let $E^\prime_\mu(s)\in M_{\delta\mu}$ be the element such that
$$
b^\prime(E^\prime_\mu(s), y)=b_\mu(s,F_\mu(y))
$$
for all $y\in M_{\delta\mu}$. Note that the element $E^\prime_\mu(s)$ exists and is unique by the non-degeneracy of $b^\prime$.  We obtain an $\SA$-linear homomorphism $E^\prime_\mu\colon S\to M_{\delta\mu}$ and can define $E_\mu:=(E^{\widetilde M}_\mu\oplus E^\prime_\mu)^T\colon M_\mu=\widetilde M_\mu\oplus S\to M_{\delta\mu}$.

We now claim that the object $M$, together with the homomorphisms $F_\mu$ and $E_\mu$ that we just constructed, is an object in $\CX_I$. The axiom (X1) is immediate. The axiom (X2)  only involves the action of $E_\mu$ on the image of $F_\mu$. As (X2) holds for $\widetilde M$, it also holds for $M$. The same argument implies that (X3a) holds. By construction,  $M_\mu/\im F_\mu=S$ is a free $\SA$-module. Hence we are left with axiom (X3b), i.e.~we have to show that the quotient $E_\mu(M_\mu)/E_\mu(\im F_\mu)$ is a torsion $\SA$-module. This means that we have to show that $E_\mu(M_\mu)$ is contained in $M_{\{\mu\},\max}$.  As $\im F_\mu=\im F^{\widetilde M}_\mu=\widetilde{M}_\mu\subset M_\mu$ we have $M_{\{\mu\},\max}=\widetilde M_{\{\mu\},\max}$. As $E_\mu(\widetilde M_\mu)=E_\mu^{\widetilde M}(\widetilde M_\mu)\subset \widetilde M_{\{\mu\},\max}$  it is sufficient to show that $E_\mu^\prime(s)$ is contained in $\widetilde{M}_{\{\mu\},\max}$ for all $s\in S$. This means that there exists some $\xi\in\SA$ and an element $z\in\widetilde{M}_{\mu}=\im F_\mu^{\widetilde M}$ such that $\xi E_\mu^\prime(s)=E_\mu(z)$. 

Let $\SK$ be the quotient field of $\SA$. We denote by $\tilde b_\SK$ the induced bilinear form on $\widetilde M_\SK=\widetilde M\otimes_\SA\SK$.  Part (3) of Lemma \ref{lemma-extbiform} shows that  the form $\tilde b_\SK$ is non-degenerate.  Hence there exists an element $x\in (\widetilde M_\mu)_\SK$ with the property that 
$$
\tilde b_\SK(x,F_\mu(y))=b^\prime_\SK(E^\prime_\mu(s),y)
$$
for all $y\in M_{\delta\mu}$ (note that by definition of $E_\mu^\prime$, the right hand side vanishes for $y$ in the kernel of $F_\mu$).  A comparison with equation $(\ast)$ shows that $E_\mu(x)=E_\mu^\prime(s)$. Now we kill denominators. Let $\xi\in\SA\setminus\{0\}$ be such that $z:=\xi x\in \widetilde M_\mu\subset (\widetilde M_\mu)_\SK$. Then $\xi E^\prime_\mu(s)=E_\mu(z)$, which proves our claim.

Finally, we need to check that the bilinear form $b$ is a non-degenerate  symmetric contravariant form on $M$. We have already checked that $b$ is symmetric and non-degenerate. For the contravariance it suffices to check that $b_\mu(x,F_\mu(y))=b_{\delta\mu}(E_\mu(x),y)$, as $b$ is an extension of $b^\prime$. For $x\in\widetilde M_\mu$ this follows from the contravariance of $\tilde b$, and for $x\in S$ this is ensured by the definition of $E_\mu^\prime$.   

So we have now constructed an object $M$ and a non-degenerate symmetric contravariant form $b$ on $M$. From the construction, property (1) immediately follows. 
Now we show that (2) holds. First we make the following observation. 
Let $x\in \ol M_\mu$. Then $\ol b(x,\ol F_\mu(y))=\ol b(\ol E_\mu(x),y)$ and hence the non-degeneracy of $\ol b$ implies that $\ol E_\mu(x)=0$ if and only if $x\in (\im \ol F_\mu)^{\perp}$. By construction of $b$ we have $(\im \ol F_\mu)^{\perp}=\ol R\subset\im \ol F_\mu$. 
Now let $f$ be an endomorphism of $M$. If $f$ is an automorphism, then its restriction $f_{I^\prime}$ to $M_{I^\prime}\cong M^\prime$ is an automorphism as well. Conversely, suppose that $f_{I^\prime}$  is an automorphism.  We want to show that this implies that $f$ is an automorphism. For this we have to show that $f_\mu$ is an automorphism. As $\SA$ is local, we can equivalently show that $\ol f_\mu$ is an automorphism.  As $\ol M_\mu$ is an $\SF$-vector space of finite dimension, it suffices to show that $\ol f_\mu$ is injective. So suppose that $x\in \ol M_\mu$ is such that $\ol f_\mu(x)=0$. Then $0=\ol E_\mu \ol f_\mu(x)=\ol f_{\delta\mu} \ol E_\mu (x)$. Our assumption implies that $\ol f_{\delta\mu}$ is an automorphism, so we deduce that $\ol E_\mu(x)=0$. By the above, this implies that $x$ is contained in the image of $\ol F_\mu$. As $\ol f_{\delta\mu}$ is an automorphism,  $\ol f_\mu$ restricts to  an automorphism on the image of $\ol F_\mu$ in $\ol M_\mu$.  Hence $x=0$, so $\ol f_\mu$ is injective. 
 \end{proof}

\subsection{Contravariant forms on maximal objects} In this section we use Proposition \ref{prop-selfdualobj} to show that there exists a non-degenerate contravariant symmetric form on $S_{\max}(\lambda)$ provided that each weight space of $S_{\max}(\lambda)$ is a free $\SA$-module.
\begin{proposition}\label{prop-extselfdual}  Suppose that $\SA$ is a local ring. Let $\lambda\in X$ and suppose that $S_{\max}(\lambda)_\mu$ is a free $\SA$-module for all $\mu\in X$. Then there exists a non-degenerate symmetric contravariant form on $S_{\max}(\lambda)$. 
\end{proposition}
\begin{remark} If $\lambda$ is dominant, then the assumption of the proposition is satisfied due to the remark following Proposition \ref{prop-domfin}.
\end{remark}
\begin{proof} For notational convenience we set $S:=S_{\max}(\lambda)$. We construct a symmetric bilinear  form $b_\mu$ on $S_\mu$ such that the direct sum $b:=\bigoplus_{\mu\in X}b_\mu$ is a symmetric contravariant bilinear form on $S$. First we set $b_\mu=0$ for all $\mu\not\le\lambda$ (we have $S_\mu=0$ in these cases). Then we fix an arbitrary non-degenerate symmetric form on the free  $\SA$-module $S_\lambda$ of rank 1.  Now suppose that we have already constructed a non-degenerate symmetric contravariant form $b^\prime$ on $S_{I^\prime}$ for some closed subset $I^\prime$, and suppose that $\mu\not\in I^\prime$ is such that $I:=I^\prime\cup\{\mu\}$ is closed as well. By construction of the maximal extension, $\mathsf{E}_{I^\prime}^I(S_{I^\prime})_\mu$ identifies with $\im F_\mu^S\subset S_\mu$. By (X3),  $\im F_\mu^S$ is a direct summand in $S_\mu$ and by assumption, $S_\mu$ is a free $\SA$-module. Since we assume that $\SA$ is local, also $\im F_\mu^S$, and hence $\mathsf{E}_{I^\prime}^I(S_{I^\prime})_\mu$ are  free $\SA$-modules. Hence we can apply Proposition \ref{prop-selfdualobj} and obtain an object $M$ in $\CX_I$  and a non-degenerate contravariant form $b$ on $M$ such that $M_{I^\prime}\cong S_{I^\prime}$ and $b_{I^\prime}\cong b^\prime$. Part (2) of Proposition \ref{prop-selfdualobj} and the fact that $S_{I^\prime}$ is indecomposable imply that $M$ is indecomposable. By Proposition \ref{prop-cform},  $M$ is maximal, hence it must be isomorphic to $S_{\max}(\mu)_I$ for some $\mu\in I$. As $M_{I^\prime}\cong S_{\max}(\lambda)_{I^\prime}$, we have $M\cong S_{\max}(\lambda)_I$. So we obtain that there exists a non-degenerate symmetric contravariant form on $S_{\max}(\lambda)_I$. We continue by induction. 
\end{proof}

\section{Tilting modules in $\CO_\SA$}

There is also the notion of a contravariant form on representations of  a quantum group. Before we come to its definition, recall that there is an  antiautomorphism $\tau$ of order $2$ on $U_\SA$ that maps $e_\alpha$ to $f_\alpha$ and $k_\alpha^{\pm 1}$ to $k_\alpha^{\pm 1}$ (this is an immediate consequence of the definition of $U_\SZ$ by generators and relations). The contravariant dual of an object $M$ in $\CO_\SA$ is given by $dM=\bigoplus_{\nu\in X} M_\nu^\ast \subset M^\ast$  with the action of $U_\SA$ twisted by the antiautomorphism $\tau$. A homomorphism $M\to dM$ is hence the same as an $\SA$-bilinear form $b\colon M\times M\to \SA$ that satisfies
\begin{itemize}
\item $b(x.m,n)=b(m,\tau(x).n)$ for all $x\in U_\SA$, $m,n\in M$.
\item $b(m,n)=0$ if $m\in M_\mu$, $n\in M_\nu$ and $\mu\ne\nu$.
\end{itemize}
Such a form is also called a {\em contravariant form on $M$}. 

Suppose that $\SA$ is a generic $\SZ$-algebra. Recall that a {\em tilting module} in $\CO_\SA$ is an object $T$ such that $T$ and its dual $dT$  admit a Weyl filtration. We denote by $\CO_\SA^{\tilt}$ the full subcategory of $\CO_\SA$ that contains all tilting modules.

\begin{theorem} Assume that $\SA$ is  local and generic.
\begin{enumerate} 
\item For any dominant weight $\lambda$,  the object $T(\lambda):=\mathsf{R}(S_{\max}(\lambda))$ is an indecomposable self-dual tilting module in $\CO$, and its endomorphism ring is local. 
\item The functors $\mathsf{R}$ and $\mathsf{S}$ induce inverse equivalences between the category of maximal objects in  $\CX^{\fin}_\SA$ and $\CO_{\SA}^{\tilt}$. 
\item If $T$ is a tilting module in $\CO_\SA$, then there are $\lambda_1,\dots,\lambda_l\in X^+$ such that $T\cong T(\lambda_1)\oplus\dots\oplus T(\lambda_l)$.
\end{enumerate}
\end{theorem}
\begin{proof} Proposition \ref{prop-extselfdual} shows that $S_{\max}(\lambda)$ admits a non-degenerate contravariant form, hence $\mathsf{R}(S_{\max}(\lambda))$ admits a non-degenerate contravariant form. In particular, $\mathsf{R}(S_{\max}(\lambda))$ is self-dual. As $\lambda$ is supposed to be dominant, $S_{\max}(\lambda)$ is an object in $\CX^{\fin}$ (cf.~Proposition \ref{prop-domfin}). Hence $T(\lambda)=\mathsf{R}(S_{\max}(\lambda))$  admits a Weyl filtration.  As it  is self-dual, it is a tilting module. 
 Its endomorphism ring is isomorphic to the endomorphism ring of $S_{\max}(\lambda)$ as   $\mathsf{R}$ is a fully faithful functor. In particular, this endomorphism ring  is local, and hence $T(\lambda)$ is indecomposable. Hence (1). 

As each maximal object in $\CX^{\fin}$ is isomorphic to a direct sum of various $S_{\max}(\lambda)$ with $\lambda$ dominant (by  Proposition \ref{prop-domfin}), (1) implies that the functor $\mathsf{R}$ maps each maximal object in $\CX^{\fin}$ to a tilting module. Conversely, let $T$ be an indecomposable tilting module in $\CO_\SA$.   As $T$ admits a Weyl filtration,  $\mathsf{S}(T)$ is an object in  $\CX^{\fin}$ by Theorem \ref{thm-Wflag}. Let $\lambda$ be a maximal weight of $T$. Then $T_\lambda$ is a free module of finite rank. Let us  fix a direct sum decomposition $\mathsf{S}(T)_\lambda=A\oplus B$, where $A$ is free of rank $1$.   By the maximality of $S_{\max}(\lambda)$ we obtain from Lemma \ref{lemma-minimalextmor} a morphism $\tilde f\colon \mathsf{S}(T) \to S_{\max}(\lambda)$ that maps $A$ isomorphically onto $S_{\max}(\lambda)_\lambda$ and $B$ to $0$. Likewise, we can find a morphism $\tilde g\colon \mathsf{S}(dT) \to S_{\max}(\lambda)$ that maps $dA$ isomorphically onto $S_{\max}(\lambda)_\lambda$ and $dB$ to $0$. Applying the functor $\mathsf{R}$ we obtain homomorphisms $f\colon T\to T(\lambda)$ and $g\colon dT \to T(\lambda)$. We now consider the dual homomorphism $dg\colon dT(\lambda)\to T$, and we fix an isomorphism $h\colon T(\lambda)\cong dT(\lambda)$ (this is possible by (1)). Then the composition $f\circ dg\circ h$ is an endomorphism of $T(\lambda)$ that is an automorphism on the highest weight space. As the endomorphism ring of $T(\lambda)$ is local, this composition is an automorphism. Hence $T(\lambda)$ is isomorphic to a direct summand of $T$. As $T$ was supposed to be indecomposable, we obtain $T\cong T(\lambda)$. As each direct summand of a tilting module is a tilting modules again, we obtain (3) by induction. Now (3) and (1) together with the fact that $\mathsf{R}$ and $\mathsf{S}$ are inverse equivalences between $\CX^{\fin}$ and $\CO^{W}$, yield (2). 
 \end{proof}

\section*{A list of notations}
The ring $\SA$ is always supposed to be a $\SZ:=\DZ[v,v^{-1}]$-module that is unital, Noetherian and a domain. In some parts of the article we assume, in addition, that $\SA$ is a local ring, or {\em generic} (see Definition \ref{def-generic}), or both. 
The basic datum is an $X$-graded $\SA$-module $M=\bigoplus_{\mu\in X} M_\mu$ together with $\SA$-linear homomorphisms $E_{\mu,\alpha,n}\colon M_\mu\to M_{\mu+n\alpha}$ and $F_{\mu,\alpha,n}\colon M_{\mu+n\alpha}\to M_\mu$ for each simple root $\alpha$ and each positive integer $n$. Then we define the following for each $\mu\in X$:

\begin{tabular}{ll}
$M_{\delta\mu}$& $:=\bigoplus_{\alpha\in\Pi,n>0}M_{\mu+n\alpha}$ \\
$E_\mu\colon M_{\mu}\to M_{\delta\mu}$ & the column vector with entries $E_{\mu,\alpha,n}$ \\
$F_\mu\colon M_{\delta\mu}\to M_\mu$ & the row vector with entries $F_{\mu,\alpha,n}$ \\
$M_{\lgl\mu\rgl}\subset M_{\delta\mu}$& the image of $E_\mu\colon M_\mu\to M_{\delta\mu}$\\
$M_{\{\mu\}}\subset M_{\lgl\mu\rgl}$ & the image of $E_\mu\circ F_\mu\colon M_{\delta\mu}\to M_{\delta\mu}$ \\
$M_{\{\mu\},\max}\subset M_{\delta\mu}$ & the preimage of the torsion part of $M_{\delta\mu}/M_{\{\mu\}}$.
\end{tabular}


\end{document}